\documentclass{amsart}

\usepackage{amssymb,amsthm,amsmath,amsxtra,pstricks}
\usepackage[all]{xy}

\newcommand{\A}{\mathbb A}

\newcommand{\N}{\mathbb N}
\newcommand{\PP}{\mathbb P}
\newcommand{\Q}{\mathbb Q}
\newcommand{\R}{\mathbb R}
\newcommand{\Z}{\mathbb Z}
\newcommand{\T}{\mathbb T}

\newcommand{\kbar}{\overline{k}}

\DeclareMathOperator{\Aut}{Aut}

\DeclareMathOperator{\conv}{conv}

\DeclareMathOperator{\inter}{int}
\DeclareMathOperator{\opchar}{char}

\DeclareMathOperator{\Proj}{Proj}
\DeclareMathOperator{\res}{res}
\DeclareMathOperator{\Spec}{Spec}
\DeclareMathOperator{\supp}{supp}
\DeclareMathOperator{\vol}{vol}

\newenvironment{enumalph}
{\begin{enumerate}}
{\end{enumerate}}
\newenvironment{enumroman}
{\begin{enumerate}}
{\end{enumerate}}

\newcommand{\calH}{\mathcal{H}}
\newcommand{\calM}{\mathcal{M}}
\newcommand{\calP}{\mathcal{P}}

\newcommand{\Mgnd}{\mathcal{M}_g^{\textup{nd}}}

\theoremstyle{plain}
\newtheorem{thm}{Theorem}[section]

\newtheorem{prop}[thm]{Proposition}

\newtheorem{lem}[thm]{Lemma}
\newtheorem{cor}[thm]{Corollary}

\newtheorem{defn}[thm]{Definition}
\newtheorem*{thmnostar}{Theorem}

\theoremstyle{remark}

\newtheorem{rmk}[thm]{Remark}
\newtheorem{exa}[thm]{Example}
\newtheorem{exm}[thm]{Example}
\newtheorem*{rmknostar}{Remark}
\newtheorem*{conven}{Conventions and notations}

\begin{document}

\title{On nondegeneracy of curves}
\author{Wouter Castryck}
\date{9 April 2008}
\address{Katholieke Universiteit Leuven, Departement Elektrotechniek (ESAT), Afdeling SCD -- COSIC, Kasteelpark Arenberg 10, B-3001 Leuven (Heverlee), Belgium}
\email{wouter.castryck@esat.kuleuven.be}
\author{John Voight}
\address{Department of Mathematics and Statistics, University of Vermont, 16 Colchester Ave, Burlington, VT 05401, USA}
\email{jvoight@gmail.com}

\begin{abstract}
We study the conditions under which an algebraic curve can be modelled
by a Laurent polynomial that is nondegenerate with respect to its Newton
polytope.
We prove that every curve of
genus $g \leq 4$ over an algebraically
closed field is nondegenerate in the above sense.  More generally, let $\Mgnd$ be the locus
of nondegenerate curves inside the moduli space of curves of genus $g \geq 2$. Then we show that $\dim
\Mgnd=\min(2g+1,3g-3)$, except for $g=7$ where $\dim
\calM_7^{\textup{nd}} = 16$; thus, a generic curve of genus $g$ is
nondegenerate if and only if $g \leq 4$.

\end{abstract}

\maketitle

Let $k$ be a perfect field with algebraic closure $\kbar$. Let $f
\in k[x^{\pm 1}, y^{\pm 1}]$ be an irreducible Laurent polynomial,
and write $f=\sum_{(i,j) \in\Z^2} c_{ij} x^iy^j$. We denote by
$\supp(f)=\{(i,j) \in \Z^2:c_{ij} \neq 0\}$ the support of $f$,
and we associate to $f$ its Newton polytope $\Delta=\Delta(f)$,
the convex hull of $\supp(f)$ in $\R^2$. We assume throughout
that $\Delta$ is $2$-dimensional. For a face $\tau \subset
\Delta$, let $f|_\tau = \sum_{(i,j) \in \tau} c_{ij} x^iy^j$.
We say that $f$ is \emph{nondegenerate} if, for every face $\tau
\subset \Delta$ (of any dimension), the system of equations
\begin{equation} \label{nondeg}
f|_\tau=x\frac{\partial f|_\tau}{\partial x} = y\frac{\partial
f|_\tau}{\partial y} = 0
\end{equation}
has no solutions in $\kbar^{*2}$.

From the perspective of toric varieties, the condition of
nondegeneracy can be rephrased as follows. The Laurent polynomial
$f$ defines a curve $U(f)$ in the torus $\T^2_k = \Spec k[x^{\pm
1}, y^{\pm 1}]$, and $\T^2_k$ embeds canonically in the
projective toric surface $X(\Delta)_k$ associated to $\Delta$ over
$k$. Let $V(f)$ be the Zariski closure of the curve $U(f)$ inside
$X(\Delta)_k$. Then $f$ is nondegenerate if and only if for every
face $\tau \subset \Delta$, we have that
$V(f) \cap \T_\tau$ is smooth of codimension $1$ in
$\T_\tau$, where $\T_\tau$ is
the toric component of $X(\Delta)_k$ associated to $\tau$.
(See Proposition~\ref{nondegen-equiv} for alternative
characterizations.)

Nondegenerate polynomials have become popular objects in explicit
algebraic geometry, owing to their connection with toric geometry
\cite{BatyrevCox}: a wealth of geometric information about $V(f)$ is
contained in the combinatorics of the Newton polytope $\Delta(f)$.
The notion was initially employed by Kouchnirenko
\cite{Kouchnirenko}, who studied nondegenerate polynomials in the
context of singularity theory.  Nondenegerate polynomials emerge
naturally in the theory of sparse resultants \cite{GKZ} and admit a
linear effective Nullstellensatz \cite[Section
2.3]{CastryckDenefVercauteren}. They make an appearance in the study
of real algebraic curves in maximal position \cite{Mikhalkin} and in
the problem of enumerating curves through a set of prescribed points
\cite{Mikhalkin2}.
In the case where $k$ is a finite field, they arise in the
construction of curves with many points
\cite{BeelenPellikaan,KreschWetherell}, in the $p$-adic cohomology
theory of Adolphson and Sperber \cite{AS}, and in explicit methods
for computing zeta functions of varieties over $k$
\cite{CastryckDenefVercauteren}. Despite their utility and seeming
ubiquity, the \emph{intrinsic} property of nondegeneracy has not
seen detailed study, with the exception of the Ph.D. thesis of
Koelman \cite{Koelman} from 1991, otherwise unpublished (see
Section~\ref{morecurves} below).

We are therefore led to the central problem of this article:
\emph{Which curves are nondegenerate?}
To the extent that toric varieties
are generalizations of projective space, this question asks us to
generalize the characterization of nonsingular plane curves amongst
all curves.
An immediate provocation for this
question was to understand the locus of curves to which the point counting algorithm of Castryck--Denef--Vercauteren \cite{CastryckDenefVercauteren} actually applies. Our results
are collected in two parts.

In the first part, comprising
Sections~\ref{sectiontechnical}--\ref{sectionsummary}, we
investigate the nondegeneracy of some interesting classes of curves (hyperelliptic, $C_{ab}$, and
low genus curves). Our conclusions can be summarized as follows.

\begin{thmnostar}
Let $V$ be a curve of genus $g$ over a perfect field $k$. Suppose
that one of the following conditions holds:
\begin{enumroman}
\item $g=0$;
\item $g=1$ and $V(k) \neq \emptyset$;
\item $g=2,3$, and either $17 \leq \# k < \infty$, or $\#k = \infty$ and $V(k) \neq \emptyset$;
\item $g=4$ and $k = \kbar$.
\end{enumroman}
Then $V$ is nondegenerate.
\end{thmnostar}

\begin{rmknostar}
\noindent The condition $\#k \geq 17$ in (iii)
ensures that $k$ is large enough to allow nontangency to the toric
boundary of $X(\Delta)_k$, but is most likely not sharp; see
Remark~\ref{project}.
\end{rmknostar}

In the second part, consisting of
Sections~\ref{SectionBounds}--\ref{morecurves}, we restrict to
algebraically closed fields $k = \kbar$ and consider the locus
$\Mgnd$ of nondegenerate curves inside the coarse moduli space of
all curves of genus $g \geq 2$. We prove the following theorem.

\begin{thmnostar}
We have $\dim \Mgnd = \min(2g + 1,3g-3)$, except for $g=7$ where
$\dim \mathcal{M}_7^{\textup{nd}}=16$.  In particular, a generic
curve of genus $g$ is nondegenerate if and only if $g \leq 4$.
\end{thmnostar}

\noindent Our methods combine ideas of Bruns--Gubeladze
\cite{BrunsGubeladze} and Haase--Schicho \cite{HaaseSchicho} and
are purely combinatorial---only the
universal property of the coarse moduli space is used.


\begin{conven}
Throughout, $\Delta \subset \R^2$ will denote a polytope with
$\dim \Delta=2$.
The coordinate functions on the ambient space $\mathbb{R}^2$ will
be denoted by $X$ and $Y$.
A \emph{facet} or \emph{edge} of a polytope is a face of dimension
$1$. A \emph{lattice polytope} is a polytope with vertices in $\mathbb{Z}^2$.
Two lattice polytopes $\Delta$ and $\Delta'$ are
\emph{equivalent} if there is an affine map
\begin{align*}
\varphi: \R^2 &\to \R^2 \\
v & \mapsto Av + b
\end{align*}
such that $\varphi(\Delta)=\Delta'$ with $A \in GL_2(\mathbb{Z})$
and $b \in \mathbb{Z}^2$.
Two Laurent polynomials $f$ and $f'$ are \emph{equivalent} if $f'$ can be obtained
from $f$ by applying such a map to the exponent vectors. Note that equivalence preserves
nondegeneracy.
For a polytope $\Delta \subset \R^2$, we let $\inter(\Delta)$
denote the interior of $\Delta$. We denote the standard
$2$-simplex in $\mathbb{R}^2$ by
$\Sigma=\conv(\{(0,0),(1,0),(0,1)\})$.
\end{conven}

\section{Nondegenerate Laurent polynomials} \label{sectiondefnondeg}

In this section, we review the geometry of
nondegenerate Laurent polynomials. We retain the notation used in the
introduction: in particular, $k$ is a perfect field, $f = \sum c_{ij} x^iy^j \in
k[x^{\pm 1}, y^{\pm 1}]$ is an irreducible Laurent polynomial, and $\Delta$ is
its Newton polytope. Our main implicit reference on toric
varieties is Fulton \cite{Fulton}.

Let $k[\Delta]$ denote the graded semigroup algebra over
$k$ generated in degree $d$ by the monomials that are supported in $d\Delta$, i.e.
\[ k[\Delta] = \bigoplus_{d=0}^\infty \langle x^iy^jt^d \, | \, (i,j) \in (d \Delta \cap \mathbb{Z}^2) \rangle_k.\]
Then
$X=X(\Delta)_k=\Proj k[\Delta]$ is the projective toric surface
associated to $\Delta$ over $k$. This surface naturally decomposes into toric components as
\[ X = \bigsqcup_{\tau \subset \Delta} \mathbb{T}_\tau, \]
where $\tau$ ranges over the faces of $\Delta$ and $\mathbb{T}_\tau \cong \mathbb{T}_k^{\dim \tau}$.
The surface $X$ is nonsingular except possibly at the zero-dimensional toric components associated to the
vertices of $\Delta$. The Laurent polynomial $f$ defines a curve in $\mathbb{T}^2_k \cong \mathbb{T}_\Delta \subset X$,
and we denote by $V=V(f)$ its closure in $X$.  Alternatively, if we denote $A=\Delta \cap \Z^2$, then $X$ can be canonically embedded in $\PP^{\#A - 1}_k = \text{Proj} \, k[t_{ij}]_{(i,j) \in A}$, and $V$ is the hyperplane section $\sum c_{ij}t_{ij} = 0$ of $X$.

We abbreviate $\partial_x=x\displaystyle{\frac{\partial}{\partial x}}$ and $\partial_y=y\displaystyle{\frac{\partial}{\partial y}}$.

\begin{defn} \label{nondegen-def}
The Laurent polynomial $f$ is \emph{nondegenerate}
if for each face $\tau \subset \Delta$, the system
\[ f|_\tau = \partial_x f|_\tau = \partial_y f|_\tau = 0 \]
has no solution in $\kbar^{*2}$.
\end{defn}

We will sometimes write that $f$ is $\Delta$-nondegenerate to emphasize that $\Delta(f) = \Delta$.

\begin{prop} \label{nondegen-equiv}
The following statements are equivalent.
\begin{enumroman}
\item $f$ is nondegenerate.
\item For each face $\tau \subset \Delta$, the ideal of $k[x^{\pm 1}, y^{\pm 1}]$ generated by
\[ f|_\tau,\partial_x f|_\tau, \partial_y f|_\tau \]
is the unit ideal.
\item For each face $\tau \subset \Delta$, the intersection $V \cap \T_\tau$ is smooth of codimension $1$ in the torus orbit $\T_\tau$ associated to $\tau$.
\item The sequence of elements $f,\partial_x f, \partial_y f$ (in degree one) forms a regular sequence in $k[\Delta]$.
\item The quotient of $k[\Delta]$ by the ideal generated by $f,\partial_x f, \partial_y f$ is finite of $k$-dimension equal to $2 \vol(\Delta)$.
\end{enumroman}
\end{prop}

\begin{rmk}
Condition (iii) can also be read as: $V$ is smooth and intersects $X \setminus \mathbb{T}^2_k$
transversally and outside the zero-dimensional toric components associated to the vertices of $\Delta$.
\end{rmk}

\begin{proof}
See Batyrev \cite[Section 4]{Batyrev} for a proof of these
equivalences and further discussion.
\end{proof}

\begin{rmk}
Some authors refer to nondegenerate as \emph{$\Delta$-regular},
though we will not employ this term.
The use of \emph{nondegenerate} to indicate a projective variety
which is not contained in a smaller projective space is
unrelated to our present usage.
\end{rmk}


\begin{exm} \label{duple}
Let $f(x,y) \in k[x,y]$ be a bivariate polynomial of degree $d \in
\Z_{\geq 1}$ with Newton polytope
$\Delta=d\Sigma=\conv(\{(0,0),(d,0),(0,d)\})$. The toric variety
$X(\Delta)_k$ is the $d$-uple Veronese embedding of $\PP^2_k$ in
$\PP^{d(d+3)/2}_k$, and $V(f)$ is the projective curve in $\mathbb{P}^2_k$ defined
by the homogenization $F(x,y,z)$ of $f$.
We see that $f(x,y)$ is $\Delta$-nondegenerate if and only if
$V(f)$ is nonsingular,
does not contain the
coordinate points $(0,0,1)$, $(0,1,0)$ and $(1,0,0)$, and is not
tangent to any coordinate axis.
\end{exm}

\begin{exm}
The following picture illustrates nondegeneracy in case of a quadrilateral Newton polytope.
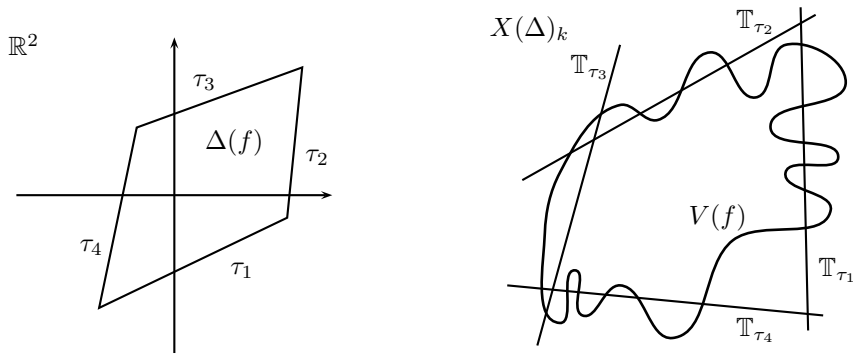
\begin{center}
\begin{pspicture}(-2.2,-2.2)(2.2,2.5)
\pspolygon[linecolor=black](-1,-1.5)(1.5,-0.3)(1.7,1.7)(-0.5,0.9)
\psline{->}(-2.1,0)(2.1,0) \psline{->}(0,-2.1)(0,2.1)
\rput(0.8,0.7){$\Delta(f)$} \rput(0.9,-1){$\tau_1$}
\rput(1.9,0.5){$\tau_2$} \rput(0.4,1.5){$\tau_3$}
\rput(-1.1,-0.7){$\tau_4$} \rput(-2,2){$\mathbb{R}^2$}
\end{pspicture} \qquad \qquad \qquad
\begin{pspicture}(-2.5,-2.2)(2.1,2.5)
\psccurve[linewidth=1pt](-1.8,-1.7)(-1.6,-1)(-1.5,-1.6)(-1,-1.2)(-0.3,-1.9)(0.5,-0.6)(1.8,-0.2)%
                        (1.2,0.3)(1.9,0.5)(1,0.9)(2,1.5)(1.2,2)(0.8,1.3)(0.2,1.9)(-0.5,1)(-1,1.2)%
                        (-1.8,0.3)(-2,-0.5)
\psline{-}(-2.5,-1.2)(1.7,-1.6) \psline{-}(1.5,-1.8)(1.4,2.5)
\psline{-}(1.6,2.4)(-2.3,0.2) \psline{-}(-1,2)(-2.1,-2)
\rput(0.3,-0.3){$V(f)$} \rput(1.9,-1){$\mathbb{T}_{\tau_1}$}
\rput(-1.4,1.7){$\mathbb{T}_{\tau_3}$}
\rput(0.8,-1.8){$\mathbb{T}_{\tau_4}$}
\rput(0.8,2.3){$\mathbb{T}_{\tau_2}$}
\rput(-2.2,2.2){$X(\Delta)_k$}
\end{pspicture}
\end{center}
\end{exm}


\begin{prop} \label{internallatticegenus}
If $f \in k[x^{\pm 1}, y^{\pm 1}]$ is nondegenerate, then there
exists a $k$-rational canonical divisor $K_\Delta$ on $V=V(f)$ such
that $\{ x^i y^j : (i,j) \in \inter(\Delta) \cap \Z^2\}$ is a $k$-basis
for the Riemann-Roch space $\mathcal{L}(K_\Delta) \subset k(V)$. In
particular, the genus of $V$ is equal to
$\#(\inter(\Delta) \cap \Z^2)$.
\end{prop}

\begin{proof}
See Khovanski\u\i\ \cite{Khovanskii} or Castryck--Denef--Vercauteren
\cite[Section 2.2]{CastryckDenefVercauteren}.
\end{proof}

\begin{rmk}
In general, if $f$ is irreducible (but not necessarily nondegenerate), one has that the
geometric genus of $V(f)$ is bounded by $\#(\inter(\Delta) \cap
\Z^2)$: this is also known as Baker's inequality \cite[Theorem 4.2]{BeelenPellikaan}.
\end{rmk}

We conclude this section with the following intrinsic definition of nondegeneracy.

\begin{defn} \label{geomnondeg}
A curve $V$ over $k$ is \emph{$\Delta$-nondegenerate} if $V$
is birational over $k$ to a curve $U \subset \mathbb{T}^2_k$ defined by a nondegenerate
Laurent polynomial $f$ with Newton polytope $\Delta$.  The curve $V$ is \emph{nondegenerate} if it is $\Delta$-nondegenerate for some $\Delta$.  The curve $V$ is \emph{geometrically nondegenerate} if
$V \times_k \kbar$ is nondegenerate over $\kbar$.
\end{defn}

\section{Moduli of nondegenerate curves}\label{sectionmoduliformulation}

We now construct the moduli space of nondegenerate
curves of given genus $g \geq 2$.  Since in this article we will be concerned with dimension estimates only,
we restrict to the case $k = \kbar$.

We denote by $\calM_g$ the coarse moduli space of curves of genus $g \geq 2$ over $k$,
with the property that for any flat family $\mathcal{V} \rightarrow M$ of curves
of genus $g$, there is a (unique) morphism $M \rightarrow \calM_g$ which maps each closed point $f \in M$
to the isomorphism class of the fiber $\mathcal{V}_f$. (See e.g.\ Mumford \cite[Theorem 5.11]{Mumford}.)


Let $\Delta \subset \R^2$ be a lattice polytope with $g$ interior lattice
points.
We will construct a flat family $\mathcal{V}(\Delta) \rightarrow M_\Delta$ which parametrizes all
$\Delta$-nondegenerate curves over $k$. The key ingredient is provided by the following result
of Gel'fand--Kapranov--Zelevinsky. Let $A = \Delta \cap \mathbb{Z}^2$ and define
the polynomial ring $R_\Delta=k[c_{ij}]_{(i,j) \in A}$.



%



\begin{prop}[{Gel'fand--Kapranov--Zelevinsky \cite{GKZ}}] \label{princadet}
There exists a polynomial $E_A \in R_\Delta$ with the property that
for any Laurent polynomial $f \in k[x^{\pm 1},y^{\pm 1}]$ with
$\supp(f) \subset \Delta$, we have that $f$ is
$\Delta$-nondegenerate if and only if $E_A(f) \neq 0$.
\end{prop}

\begin{proof}
The proof of Gel'fand--Kapranov--Zelevinsky \cite[Chapter 10]{GKZ} is over $\mathbb{C}$; however, the
construction yields a polynomial over $\Z$ which is easily seen to
characterize nondegeneracy for an arbitrary (algebraically closed) field.
\end{proof}

The polynomial $E_A$ is known as the \emph{principal
$A$-determinant} and is given by the
\emph{$A$-resultant} $\res_A(F,\partial_1 F, \partial_2 F)$.
It is homogeneous in the variables $c_{ij}$ of
degree $6\vol(\Delta)$, and its irreducible factors are the
\emph{face discriminants} $D_\tau$ for faces $\tau \subset \Delta$.

\begin{exm}
Consider the universal plane conic
\[ F=c_{00}+c_{10}x+c_{01}y+c_{20}x^2+c_{11}xy+c_{02}y^2, \]
associated to the Newton polytope $2\Sigma$ as in Example \ref{duple}.

Then
\[ E_A=c_{00}c_{02}c_{20}(c_{11}^2-4c_{02}c_{20})(c_{10}^2-4c_{00}c_{20})(c_{01}^2-4c_{00}c_{02})D_\Delta \]
where
\[ D_\Delta=4c_{00}c_{20}c_{02}-c_{00}c_{11}^2 - c_{10}^2c_{02} - c_{01}^2c_{20}+c_{10}c_{01}c_{11}. \]
The nonvanishing of the factor $c_{00}c_{02}c_{20}$ (corresponding to the discriminants of the zero-dimensional faces) ensures that the curve does not contain a coordinate point, and in particular does not have Newton polytope smaller than $2\Sigma$; the nonvanishing of the quadratic factors (corresponding to the one-dimensional faces) ensures that the curve intersects the coordinate lines in two distinct points; and the nonvanishing of $D_\Delta$ ensures that the curve is smooth.
\end{exm}


Let $M_\Delta$ be the complement in $\mathbb{P}^{\# A - 1}_k = \text{Proj} \, R_\Delta$ of the algebraic set defined by $E_A$. By the above,
$M_\Delta$ parameterizes nondegenerate polynomials having $\Delta$ as Newton polytope. One can show that
\begin{equation} \label{dimensionremark}
 \dim M_\Delta = \# A - 1,
\end{equation}
which is a non-trivial statement if $k$ is of finite characteristic (and false in general for an arbitrary number of variables), see \cite[Section~2]{CastryckDenefVercauteren}. Let $\mathcal{V}(\Delta)$ be the closed subvariety of
\[ X(\Delta)_k \times M_\Delta \subset \text{Proj} \, k[t_{ij}] \times \text{Proj} \, k[c_{ij}] \]
defined by the universal hyperplane section
\[ \sum_{(i,j) \in A} c_{ij} t_{ij} = 0. \]
Then the universal family of $\Delta$-nondegenerate curves is realized by
the projection map $\varphi: \mathcal{V}(\Delta) \rightarrow M_\Delta$. The fiber $\mathcal{V}(\Delta)_f$
above a nondegenerate Laurent polynomial $f \in M_\Delta$ is precisely
the corresponding curve $V(f)$, realized as the corresponding hyperplane section of
$X(\Delta)_k \subset \text{Proj} \, k[t_{ij}]$. Note that $\varphi$ is
indeed flat \cite[Theorem III.9.9]{Hartshorne}, since the Hilbert polynomial of $\mathcal{V}(\Delta)_f$
is independent of $f$: its degree is equal to $\deg X(\Delta)_k$ and its genus is $g$ by Proposition~\ref{internallatticegenus}.

Thus by the universal property of $\calM_g$, there is a morphism $h_\Delta : M_\Delta \rightarrow \mathcal{M}_g$, the image of
which consists precisely of all isomorphism classes containing a $\Delta$-nondegenerate curve. Let $\mathcal{M}_\Delta$ denote the Zariski closure of the image of $h_\Delta$.  Finally, let
\[ \Mgnd = \bigcup_{g(\Delta) = g} \calM_\Delta, \]
where the union is taken over all polytopes $\Delta$ with $g$ interior
lattice points, of which there are finitely many up to equivalence (see Hensley \cite{Hensley}).

The aim of Sections~\ref{SectionBounds}--\ref{morecurves} is to estimate $\dim \Mgnd$. This is done by first refining
the obvious upper bounds $\dim \mathcal{M}_\Delta \leq \dim M_\Delta = \#(\Delta \cap \mathbb{Z}^2) - 1$,
taking into account the action of the automorphism group $\text{Aut}(X(\Delta)_k)$, and then estimating the outcome in terms of $g$.

\begin{rmk}
It follows from the fact that $\calM_g$ is of general type for $g \geq 23$ (see e.g.\ \cite{HarrisMorrison}) that $\dim \Mgnd < \dim \calM_g=3g-3$ for $g \geq 23$, since each component of $\Mgnd$ is rational.  Below, we obtain much sharper results which do not rely on this deep result.
\end{rmk}

\section{Triangular nondegeneracy} \label{sectiontechnical}

In Sections~\ref{SectionInterestingClasses}--\ref{sectionthreefour},
we study the nondegeneracy of certain well-known classes, such as elliptic, hyperelliptic
and $C_{ab}$ curves.
In many cases, classical constructions provide models for these
curves that are supported on a triangular Newton polytope; the elementary
observations in this section will allow us to prove that these
models are nondegenerate when $\#k$ is not too small.

\begin{lem} \label{translating}
Let $f(x,y) \in k[x,y]$ define a smooth affine curve of genus $g$
and suppose that $\# k > 2(g + \max(\deg_x f, \deg_y f)-1) +
\min(\deg_x f, \deg_y f)$. Then there exist $x_0,y_0 \in k$ such
that the translated curve $f(x-x_0,y-y_0)$ does not contain $(0,0)$
and is also nontangent to both the $x$- and the $y$-axis.
\end{lem}

\begin{proof}
Suppose $\deg_yf \leq \deg_xf$. Applying the Riemann-Hurwitz theorem
to the projection map $(x,y) \mapsto x$, one verifies that there are
at most $2(g + \deg_y f - 1)$ points with a vertical tangent.
Therefore, we can find an $x_0 \in k$ such that $f(x-x_0,y)$ is
nontangent to the $y$-axis. Subsequently, there are at most $2(g +
\deg_x f - 1) + \deg_yf$ values of $y_0 \in k$ for which
$f(x-x_0,y-y_0)$ is tangent to the $x$-axis and/or contains $(0,0)$.
\end{proof}

\begin{lem} \label{fieldlowerbound}
Let $a \leq b \in \mathbb{Z}_{\geq 2}$ be such that $\gcd(a,b) \in
\{1,a\}$, and let $\Delta$ be the triangular lattice polytope
$\conv(\{(0,0),(b,0),(0,a)\})$. Let $f(x,y) \in k[x,y]$ be an
irreducible polynomial such that:
\begin{itemize}
 \item $f$ is supported on $\Delta$, and
 \item the genus of $V(f)$ equals $g = \#(\inter(\Delta) \cap
\mathbb{Z}^2)$.
\end{itemize}
Then if $\#k > 2(g+b-1) + a$, we have that $V(f)$ is
$\Delta$-nondegenerate.
\end{lem}

\begin{proof}
First suppose that $\gcd(a,b)=1$. The coefficients of $x^b$ and
$y^a$ must be nonzero, because else $\#(\inter(\Delta(f)) \cap \Z^2) < g$,
which contradicts Baker's inequality. For the same reason, $f$ must
define a smooth affine curve: if $(x_0,y_0)$ is a singular point
(over $\kbar$), then $\#(\inter(\Delta(f(x-x_0,y-y_0)) \cap \Z^2)) <
g$. The result now follows from Lemma~\ref{translating}. Note that
the nonvanishing of the face discriminant $D_\tau$, where $\tau$ is
the edge connecting $(b,0)$ and $(0,a)$, follows automatically from
the fact that $\tau$ has no interior lattice points.

Next, suppose that $\gcd(a,b) = a$. If $a < b$ then the coefficient of $y^a$ must be nonzero, and
by substituting $y \leftarrow y + x^{b/a}$ if necessary we may assume that the coefficient of $x^b$
is nonzero as well. If $a=b$ then the coefficient of $y^a$ might be zero, but then the coefficient of $x^b$ is nonzero and
interchanging the role of $x$ and $y$
solves this problem. As above, we have that $f$ defines a smooth affine
curve. So by applying Lemma~\ref{translating}, we may assume that
the face discriminants decomposing $E_{\Delta \cap \Z^2}$ are
nonvanishing at $f$, with the possible exception of $D_\tau$, where
$\tau$ is the edge connecting $(b,0)$ and $(0,a)$. However, under
the equivalence
\[ \mathbb{R}^2 \to \mathbb{R}^2 : (X,Y) \mapsto (b - X -
\frac{b}{a} Y, Y), \] $\tau$ is interchanged with the edge
connecting $(0,0)$ and $(0,a)$. By applying Lemma~\ref{translating}
again, we obtain full nondegeneracy.
\end{proof}

\section{Nondegeneracy of curves of genus at most one} \label{SectionInterestingClasses}

\subsection*{Curves of genus $\mathbf{0}$} \label{genus0}
Let $V$ be a curve of genus $0$ over $k$.  The anticanonical divisor embeds
$V \hookrightarrow \PP^2_k$ as a smooth conic.  If $\#k=\infty$, then
by Lemma \ref{fieldlowerbound} and Proposition \ref{nondegen-equiv}, we see that $V$ is nondegenerate.
If $\#k < \infty$ then $V(k) \neq \emptyset$ by Wedderburn, hence $V \cong \PP^1_k$ can be embedded as a nondegenerate line in $\PP^2$.  Therefore, any curve $V$ of genus $0$ is $\Delta$-nondegenerate, where $\Delta$ is one of the following:
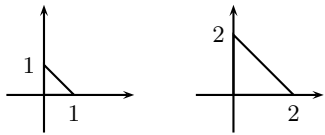
\begin{center}
\begin{pspicture}(-0.5,-0.5)(1.2,1.2)
\pspolygon[fillstyle=solid,linecolor=black](0,0)(0.4,0)(0,0.4)
\psline{->}(-0.5,0)(1.2,0) \psline{->}(0,-0.5)(0,1.2)
\rput(-0.2,0.4){\small $1$} \rput(0.4,-0.25){\small $1$}
\end{pspicture}
\qquad
\begin{pspicture}(-0.5,-0.5)(1.2,1.2)
\pspolygon[fillstyle=solid,linecolor=black](0,0)(0.8,0)(0,0.8)
\psline{->}(-0.5,0)(1.2,0) \psline{->}(0,-0.5)(0,1.2)
\rput(-0.2,0.8){\small $2$} \rput(0.8,-0.25){\small $2$}
\end{pspicture}
\end{center}

\subsection*{Curves of genus $\mathbf{1}$}
Let $V$ be a curve of genus $1$ over $k$.  First suppose that $V(k)
\neq \emptyset$. Then $V$ is an elliptic curve and hence can be defined by a
nonsingular Weierstrass equation
\begin{equation} \label{weierstrass-elliptic}
y^2 + a_1xy+a_3y = x^3 + a_2x^2 + a_4x + a_6
\end{equation}
with $a_i \in k$. The corresponding Newton polytope $\Delta$ is
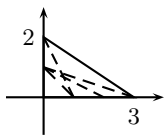
\begin{center}
\begin{pspicture}(-0.5,-0.5)(1.6,1.2)
\pspolygon[fillstyle=solid,linecolor=black](0,0)(1.2,0)(0,0.8)
\psline{->}(-0.5,0)(1.6,0) \psline{->}(0,-0.5)(0,1.2)
\psline[linestyle=dashed]{-}(0,0.8)(0.4,0)
\psline[linestyle=dashed]{-}(0,0.4)(0.4,0)
\psline[linestyle=dashed]{-}(0,0.4)(0.8,0)
\psline[linestyle=dashed]{-}(0,0.4)(1.2,0) \rput(-0.2,0.8){\small
$2$} \rput(1.2,-0.25){\small $3$}
\end{pspicture}
\end{center}
where one of the dashed lines appears as a facet if $a_6 = 0$. By
Lemma~\ref{fieldlowerbound}, we have that $V$ is nondegenerate if
$\#k \geq 9$. With some extra work we can get rid of this condition.

For $A=\Delta \cap \Z^2$, the principal $A$-determinant has $7$ or
$9$ face discriminants $D_\tau$ as irreducible factors.  The
nonvanishing of $D_\Delta$ corresponds to the fact that our curve is smooth in $\T_k^2$. In case $\tau$
is a vertex or a facet containing no interior lattice points, the
nonvanishing of $D_\tau$ is automatic. Thus it suffices to consider
the discriminants $D_\tau$ for $\tau$ a facet supported on the
$X$-axis (denoted $\tau_X$) or the $Y$-axis (denoted $\tau_Y$).
First, suppose that $\opchar k \neq 2$. After completing the square,
we have $a_1=a_3=0$ and the nonvanishing of $D_{\tau_X}$ follows
from the fact that the polynomial $p(x) = x^3+a_2x^2+a_4x+a_6$ is
squarefree. The nonvanishing of $D_{\tau_Y}$ (if $\tau_Y$ exists) is
clear. Now suppose $\opchar k = 2$.  Let $\delta$ be the number of distinct roots (over
$\kbar$) of $p(x) = x^3+a_2x^2+a_4x+a_6$. If $\delta = 3$ then
$D_{\tau_X}$ is non-vanishing. For the nonvanishing of $D_{\tau_Y}$,
it then suffices to substitute $x \leftarrow x + 1$ if necessary, so that $a_3$ is nonzero (note that
not both $a_1$ and $a_3$ can be zero). If $\delta < 3$ then $p(x)$ has a
root $x_0$ of multiplicity at least $2$. Since $k$ is perfect, this
root is $k$-rational and after substituting $x \mapsto x + x_0$ we
have $p(x) = x^3 + a_2x^2$. In particular, $D_{\tau_X}$ (if $\tau_X$
exists) and $D_{\tau_Y}$ do not vanish.

In conclusion, we have shown that every genus $1$ curve $V$ over a
field $k$ is nondegenerate, given that $V(k) \neq \emptyset$. This
condition is automatically satisfied if $k$ is a finite field (by
Hasse--Weil) or if $k$ is algebraically closed. In particular, every
genus $1$ curve is geometrically nondegenerate. More generally, we
define the \emph{index} of a curve $V$ over a field $k$ to be the
least degree of an effective non-zero $k$-rational divisor on $V$
(equivalently, the least extension degree of a field $L \supset k$
for which $V(L) \neq \emptyset$).  We then have the following
criterion.


\begin{lem} \label{genus1nondegcriterion}
A curve $V$ of genus $1$ is nondegenerate if and only if $V$ has index at most $3$.
\end{lem}

\begin{proof}
First, assume that $V$ is nondegenerate. There are exactly $16$
equivalence classes of polytopes with only $1$ interior lattice
point; see \cite[Figure 2]{PoonenRodriguez} or the appendix at the
end of this article. So we may assume that $V$ is
$\Delta$-nondegenerate with $\Delta$ in this list. Now for every
facet $\tau \subset \Delta$, the toric component $\T_\tau$ of
$X(\Delta)_k$ cuts out an effective $k$-rational divisor of degree
$\ell(\tau)$ on $V$, where $\ell(\tau) + 1$ is the number of lattice
points on $\tau$.  The result then follows, since one easily
verifies that every polytope in the list contains a facet $\tau$
with $\ell(\tau) \leq 3$.

Conversely, suppose that $V$ has index $\imath \leq 3$. If $\imath =
1$, we have shown above that $V$ is nondegenerate.  If $\imath = 2$ (resp.\ $\imath=3$), using Riemann-Roch one can
construct a plane model $f \in k[x,y]$ with $\Delta(f) \subset \conv(\{(0,0),(4,0),(0,2)\})$ (resp.\ $\Delta(f) \subset 3\Sigma$); see e.g.\ Fisher \cite[Section 3]{Fisher} for details.  Then since $V(k) = \emptyset$ and hence $\#k = \infty$, an application of Lemma \ref{fieldlowerbound} concludes the proof.
\end{proof}


\begin{rmk} \label{ellipticnotnondeg}
There exist genus $1$ curves of arbitrarily large index over every
number field; see Clark \cite{Clark}. Hence there exist infinitely
many genus $1$ curves which are not nondegenerate.
\end{rmk}

\section{Nondegeneracy of hyperelliptic curves and $C_{ab}$ curves} \label{sectionhyperelliptic}

\subsection*{Hyperelliptic curves} A curve $V$ over $k$ of genus $g \geq 2$ is \emph{hyperelliptic} if there
exists a nonconstant morphism $V \to \PP^1_k$ of degree $2$. The morphism is automatically
separable \cite[Proposition IV.2.5]{Hartshorne}
and the curve can be defined by a Weierstrass equation
\begin{equation} \label{weierstrass}
 y^2 + q(x)y = p(x).
 \end{equation}
Here $p(x),q(x) \in k[x]$ satisfy $2 \deg q(x) \leq \deg p(x)$ and $\deg p(x) \in \{2g+1, 2g+2\}$.
The universal such curve has Newton
polytope
as follows:
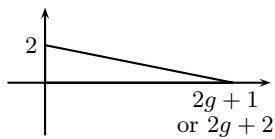
\begin{center}
\begin{pspicture}(-0.5,-0.7)(3.0,1)
\pspolygon[fillstyle=solid,linecolor=black](0,0)(2.5,0)(0,0.5)
\psline{->}(-0.5,0)(3.0,0) \psline{->}(0,-0.7)(0,1)
\rput(-0.18,0.5){\small $2$} \rput(2.4,-0.25){\small $2g+1$}
\rput(2.4,-0.58){\small \text{or }$2g+2$}
\end{pspicture}
\end{center}
By Lemma~\ref{fieldlowerbound}, if $\#k \geq 6g + 5$ then $V$ is
nondegenerate. In particular, if $\#k \geq 17$ then every curve of genus $2$ is nondegenerate.

If $\opchar k \neq 2$, we can drop the condition on
$\#k$ by completing the square, as in the elliptic curve case. This observation immediately
weakens the condition to $\#k \geq 2^{ \lfloor \log_2(6g + 5) \rfloor } + 1$. As a consequence, $\#k \geq 17$ is
also sufficient for every hyperelliptic curve of genus $3$ or $4$ to be nondegenerate.

Conversely, any curve defined by a nondegenerate polynomial as
in (\ref{weierstrass}) is hyperelliptic.
We conclude that $\dim \calM_\Delta= \dim \calH_g = 2g - 1$ \cite[Example IV.5.5.5]{Hartshorne}.

One can decide if a nondegenerate polynomial $f$ defines a
hyperelliptic curve according to the following criterion, which also
appears in Koelman \cite[Lemma 3.2.9]{Koelman} with a more
complicated proof.

\begin{lem} \label{hyperelliptic}
Let $f \in k[x^{\pm},y^{\pm}]$ be nondegenerate and suppose $\#\inter(\Delta(f)\cap\Z^2)\geq 2$. Then $V(f)$ is hyperelliptic if and only if the interior
lattice points of $\Delta(f)$ are collinear.
\end{lem}


\begin{proof}
We may assume that $\Delta = \Delta(f)$ has $g \geq 3$ interior lattice points,
since all curves of genus $2$ are hyperelliptic and any two points are collinear.

Let $L \subset k(V)$ be the subfield generated by all quotients of
functions in $\mathcal{L}(K)$, where $K$ is a canonical
divisor on $V$. Then $L$ does not depend on the choice of $K$,
and $L$ is isomorphic to the rational function field
$k(\mathbb{P}^1_k)$ if and only if $V$ is hyperelliptic.

We now show that $L \cong k(\PP^1_k)$ if and only if the interior
lattice points of $\Delta$ are collinear.  We may assume that
$(0,0)$ is in the interior of $\Delta$.  Then from Proposition
\ref{internallatticegenus}, we see that $L$ contains all monomials
of the form $x^iy^j$ for $(i,j) \in \inter(\Delta) \cap \Z^2$.  In
particular, if the interior lattice points of $\Delta$ are not
collinear then after a transformation we may assume
further that $(0,1),(1,0) \in \inter(\Delta)$, whence $L \supset
k(x,y)=k(V)$; and if they are collinear, then clearly $L \cong
k(\PP_k^1)$.  The result then follows.
\end{proof}

For this reason, we call a lattice polytope \emph{hyperelliptic}
if its interior lattice points are collinear.

A curve $V$ over $k$ of genus $g \geq 2$ is called
\emph{geometrically hyperelliptic} if $V_{\kbar} = V \times_k \kbar$
is hyperelliptic.  Every hyperelliptic curve is geometrically
hyperelliptic, but not conversely: if $V \to C \subset \PP^{g-1}_k$
is the canonical morphism, then $V$ is hyperelliptic if and only if
$C \cong \PP^1_k$. This latter condition is satisfied if and only if
$C(k) \neq \emptyset$, which is guaranteed when $k$ is finite,
when $V(k) \neq \emptyset$, and when $g$ is even.

\begin{lem} \label{geometricallyhyperellipticbutnotnondegenerate}
Let $V$ be a geometrically hyperelliptic curve which is nonhyperelliptic.
Then $V$ is not nondegenerate.
\end{lem}

\begin{proof}
Suppose that $V$ is geometrically hyperelliptic and
$\Delta$-nondegenerate for some lattice polytope $\Delta$.  Then
applying Lemma~\ref{hyperelliptic} to $V_{\kbar}$, we see that the
interior lattice points of $\Delta$ are collinear. But then again by
Lemma~\ref{hyperelliptic} (now applied to $V$ itself), $V$ must be
hyperelliptic.
\end{proof}




\subsection*{$\mathbf{C_{ab}}$ curves}
Let $a,b \in \Z_{\geq 2}$ be coprime.  A \emph{$C_{ab}$ curve} is a
curve having a place with Weierstrass semigroup $a\Z_{\geq 0} +
b\Z_{\geq 0}$ (see Miura \cite{Miura}).
Any $C_{ab}$ curve is defined by a Weierstrass equation
\begin{equation} \label{cab}
f(x,y) = \sum_{\substack{i,j \in \N \\ ai + bj \leq ab}} c_{ij}
x^iy^j=0.
\end{equation}
with $c_{0a},c_{b0} \neq 0$.  By Lemma~\ref{fieldlowerbound}, if $\#k \geq 2(g + a + b - 2)$ then
we may assume that this polynomial is nondegenerate with respect to
its Newton polytope $\Delta_{ab}$:
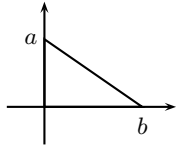
\begin{center}
\begin{pspicture}(-0.5,-0.5)(1.5,1.5)
\pspolygon[fillstyle=solid,linecolor=black](0,0)(1.3,0)(0,0.9)
\psline{->}(-0.5,0)(1.75,0) \psline{->}(0,-0.5)(0,1.4)
\rput(-0.18,0.9){\small $a$} \rput(1.3,-0.25){\small $b$}
\end{pspicture}
\end{center}
Conversely, every curve given by a $\Delta_{ab}$-nondegenerate
polynomial is $C_{ab}$, and the unique place dominating the point at
projective infinity has Weierstrass semigroup $a\Z_{\geq
2}+b\Z_{\geq 2}$ (see Matsumoto \cite{Matsumoto}). Note that if $k$
is algebraically closed, the class of hyperelliptic curves of genus
$g$ coincides with the class of $C_{2,2g+1}$ curves.

The moduli space of all $C_{ab}$ curves (for varying $a$ and $b$) of
fixed genus $g$ is then a finite union of moduli spaces
$\calM_{\Delta_{ab}}$. One can show that its dimension equals $2g-1$
by an analysis of the Weierstrass semigroup, which has been done in
Rim--Vitulli \cite[Corollary~6.3]{RimVitulli}. This dimension equals
$\dim \calH_g = \dim \calM_{\Delta_{2,2g+1}}$ and in fact this is
the dominating part: in Example~\ref{brunsexamplecab} we will show
that $\dim \calM_{\Delta_{ab}} < 2g - 1$ if $a,b\geq 3$ and $g \geq
6$.



\section{Nondegeneracy of curves of genus three and four} \label{sectionthreefour}

\subsection*{Curves of genus $\mathbf{3}$}

A genus $3$ curve $V$ over $k$ is either geometrically hyperelliptic or
it canonically embeds in $\mathbb{P}^2_k$ as a plane quartic.

If $V$ is geometrically hyperelliptic, then $V$ may not be
hyperelliptic and hence (by
Lemma~\ref{geometricallyhyperellipticbutnotnondegenerate}) not
nondegenerate. For example, over $\Q$ there exist degree $2$ covers
of the imaginary circle having genus $3$. However, if $k$ is
finite or $V(k) \neq 0$ then every geometrically
hyperelliptic curve is hyperelliptic. If moreover $\# k \geq 17$
we can conclude that $V$ is nondegenerate. See
Section~\ref{sectionhyperelliptic} for more details.

If $V$ is embedded as a plane quartic, then assuming $\#k \geq 17$,
we can apply Lemma~\ref{fieldlowerbound} and see that $V$ is defined
by a $4\Sigma$-nondegenerate Laurent polynomial.

\subsection*{Curves of genus $\mathbf{4}$}

Let $V$ be a curve of genus $4$ over $k$. If $V$ is a geometrically
hyperelliptic curve then it is hyperelliptic, since the genus is
even; thus if $\# k \geq 17$ then $V$ is nondegenerate (see
Section~\ref{sectionhyperelliptic}). Assume therefore that $V$ is
nonhyperelliptic. Then it canonically embeds as a curve of degree
$6$ in $\PP^3_k$ which is the complete intersection of a unique
quadric surface $Q$ and a (non-unique) cubic surface $C$
\cite[Example IV.5.2.2]{Hartshorne}.

First, we note that if $V$ is $\Delta$-nondegenerate for some
nonhyperelliptic lattice polytope $\Delta \subset \R^2$, then $Q$ or
$C$ must have combinatorial origins as follows. Let
$\Delta^{(1)}=\conv(\inter(\Delta) \cap \Z^2)$. Up to equivalence,
there are three possible arrangements for these interior lattice
points:
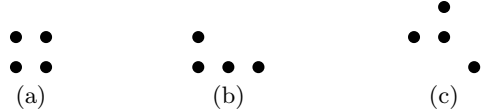
\begin{center}
\begin{pspicture}(-1,-0.5)(1,1)
\pscircle[fillstyle=solid,fillcolor=black](0,0){0.08}
\pscircle[fillstyle=solid,fillcolor=black](0.4,0){0.08}
\pscircle[fillstyle=solid,fillcolor=black](0,0.4){0.08}
\pscircle[fillstyle=solid,fillcolor=black](0.4,0.4){0.08}
\rput(0.2,-0.4){\small (a)}
\end{pspicture}
\qquad
\begin{pspicture}(-1,-0.5)(1,1)
\pscircle[fillstyle=solid,fillcolor=black](0,0){0.08}
\pscircle[fillstyle=solid,fillcolor=black](0.4,0){0.08}
\pscircle[fillstyle=solid,fillcolor=black](-0.4,0){0.08}
\pscircle[fillstyle=solid,fillcolor=black](-0.4,0.4){0.08}
\rput(0,-0.4){\small (b)}
\end{pspicture}
\quad
\begin{pspicture}(-1,-0.5)(1,1)
\pscircle[fillstyle=solid,fillcolor=black](0,0.4){0.08}
\pscircle[fillstyle=solid,fillcolor=black](0.4,0.4){0.08}
\pscircle[fillstyle=solid,fillcolor=black](0.4,0.8){0.08}
\pscircle[fillstyle=solid,fillcolor=black](0.8,0){0.08}
\rput(0.4,-0.4){\small (c)}
\end{pspicture}
\end{center}
By Proposition \ref{internallatticegenus}, one verifies that $V$
canonically maps to $X^{(1)}=X(\Delta^{(1)})_k \subset
\mathbb{P}^3_k$. In (a), $X^{(1)}$ is nothing else but the Segre
product $\mathbb{P}^1_k \times \mathbb{P}^1_k$ defined by the
equation $xz = yw$ in $\mathbb{P}^3_k$, and by uniqueness it must
equal $Q$. For (b), $X^{(1)}$ is the singular quadric cone $yz =
w^2$, which again must equal $Q$. For (c), $X^{(1)}$ is the singular
cubic $xyz=w^3$, which must be an instance of $C$. Note that a curve
$V$ can be $\Delta$-nondegenerate with $\Delta^{(1)}$ as in (a) or
(b), but not both: whether $Q$ is smooth or not is intrinsic, since
$Q$ is unique. The third type (c) is special, and we leave it as an
exercise to show that the locus of curves of genus $4$ which
canonically lie on such a singular cubic surface is a codimension
$\geq 2$ subspace of $\calM_4$ (use the dimension bounds from
Section~\ref{SectionBounds}).

With these observations in mind, we work towards conditions under
which our given nonhyperelliptic genus $4$ curve $V$ is
nondegenerate. Suppose first that the quadric $Q$ has a (necessarily
$k$-rational) singular point $T$; then $V$ is called \emph{conical}.
This corresponds to the case where $V_{\kbar} = V \times_k \kbar$ has a unique $g^1_3$,
and represents a codimension 1 subscheme of $\calM_4$ \cite[Exercise IV.5.3]{Hartshorne}. If $Q(k)=\{T\}$
then $V$ cannot be nondegenerate with respect to any polytope with
$\Delta^{(1)}$ as in (a) or (b), since then $Q$ is not isomorphic to
either of the corresponding canonical quadric surfaces $X^{(1)}$. If
$Q(k) \supsetneq \{T\}$, which is guaranteed if $k$ is finite or if $V(k) \neq \emptyset$, then after a
choice of coordinates we can identify $Q$ with the weighted
projective space $\PP(1,2,1)$. Our degree $6$ curve $V$ then has an
equation of the form
\[ f(x,y,z) = y^3+a_2(x,z)y^2+a_4(x,z)y+a_6(x,z) \]
with $\deg a_i=i$; the equation is monic in $y$ because $T \not\in
V$. By Lemma~\ref{fieldlowerbound}, if $\#k \geq 23$ then we may
assume that $f(x,y,1)$ is nondegenerate with respect to its Newton
polytope $\Delta$ as follows:
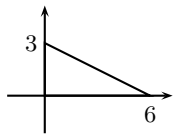
\begin{center}
\begin{pspicture}(-0.5,-0.5)(1.7,1.2)
\pspolygon[fillstyle=solid,linecolor=black](0,0)(1.4,0)(0,0.7)
\psline{->}(-0.5,0)(1.7,0) \psline{->}(0,-0.5)(0,1.2)
\rput(-0.18,0.7){\small $3$} \rput(1.4,-0.25){\small $6$}
\end{pspicture}
\end{center}

Next, suppose that $Q$ is smooth; then $V$ is called
\emph{hyperboloidal}. This corresponds to the case where $V_{\kbar}$
has two $g^1_3$'s, and represents a dense subscheme of $\calM_4$
\cite[Exercise IV.5.3]{Hartshorne}.  If $Q \not\cong \PP^1_k \times
\PP^1_k$ (e.g. this will be the case whenever the discriminant of
$Q$ is nonsquare), then again $V$ cannot be nondegenerate with
respect to $\Delta$ with $\Delta^{(1)}$ as in (a) or (b). Therefore
suppose that $k$ is algebraically closed. Then $Q \cong
\mathbb{P}^2_k \times \mathbb{P}^2_k$ and $V$ can be projected to a
plane quintic with $2$ nodes \cite[Exercise IV.5.4]{Hartshorne}.

Consider the line connecting these nodes. Generically, it will intersect the nodes with multiplicity $2$, i.e. it will intersect
all branches transversally. By Bezout, the line will then intersect the curve transversally in one
other point.
This observation fits within
the following general phenomenon. Let $d \in \Z_{\geq 4}$, and
consider the polytope $\Delta=d\Sigma$ with up to three of its
angles pruned as follows:
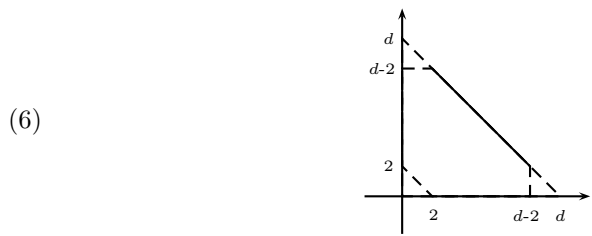
\begin{equation} \label{prunesimplex}
\begin{minipage}[c]{1.25in}
\begin{pspicture}(-0.5,-0.5)(2.5,2.5)
\pspolygon[fillstyle=solid,linestyle=dashed](0,0.4)(0,1.7)(0.4,1.7)(1.7,0.4)(1.7,0)(0.4,0)
\pspolygon[fillstyle=none,linestyle=dashed](0,0)(2.1,0)(0,2.1)
\psline{-}(1.7,0.4)(0.4,1.7) \psline{->}(-0.5,0)(2.5,0)
\psline{->}(0,-0.5)(0,2.5) \rput(-0.18,2.1){\tiny $d$}
\rput(-0.27,1.7){\tiny $d$-$2$} \rput(-0.18,0.4){\tiny $2$}
\rput(2.1,-0.25){\tiny $d$} \rput(1.65,-0.25){\tiny $d$-$2$}
\rput(0.42,-0.25){\tiny $2$}
\end{pspicture}
\end{minipage}
\end{equation}
Let $f \in k[x,y]$ be a nondegenerate polynomial with Newton
polytope $\Delta$. If we prune no angle of $d\Sigma$, then
$X(\Delta)_k \cong \PP^2_k$ (it is the image of the $d$-uple
embedding) and $V(f)$ is a smooth plane curve of degree $d$. Pruning
an angle has the effect of blowing up $X(\Delta)_k$ at a coordinate
point; the image of $V(f)$ under the natural projection $X(\Delta)_k
\to \PP^2_k$ has a node at that point.  If we prune $m=2$ (resp.\ $m=3$) angles, then we likewise
obtain the blow-up of $\PP^2_k$ at $m$ points and the image of
$V(f)$ in $\PP^2_k$ has $m$ nodes.  Since $f$ is nondegenerate, the
line connecting any two of these nodes intersects the curve
transversally elsewhere, and due to the shape of $\Delta$
the intersection multiplicity at the nodes will be $2$. Conversely, every projective plane curve
having at most $3$ nodes such that the line connecting any two nodes
intersects the curve transversally (also at the nodes themselves), is nondegenerate. Indeed, after
an appropriate projective transformation, it will have a Newton
polytope as in (\ref{prunesimplex}). In particular, our hyperboloidal genus 4 curve $V$ will
be $\Delta$-nondegenerate, where $\Delta$ equals polytope (h.1) from Section~\ref{sectionsummary} below.

Exceptionally, the line connecting the two nodes of our quintic may be tangent to one of the
branches at a node. Using a similar reasoning, we conclude that $V$ is $\Delta$-nondegenerate, with
$\Delta$ equal to polytope (h.2) from Section~\ref{sectionsummary} below.

\begin{rmk} \label{genus4notnondeg}
As in Remark~\ref{ellipticnotnondeg}, an argument based on the index
shows that there exist genus $4$ curves which are not nondegenerate.
A result by Clark \cite{Clark2} states that for every $g \geq 2$,
there exists a number field $k$ and a genus $g$ curve $V$ over $k$,
such that the index of $V$ is equal to $2g-2$, the degree of the
canonical divisor.  In particular, there exists a genus $4$ curve
$V$ of index $6$. Such a curve cannot be nondegenerate. Indeed, for
each of the above arrangements (a)--(c), $X^{(1)}$ contains the line
$z = w = 0$, which cuts out an effective divisor on $V$ of degree
$3$ in cases (a) and (b) and degree $2$ in case (c).
\end{rmk}



\section{Nondegeneracy of low genus curves: summary}
\label{sectionsummary}

We now summarize the results of the preceding sections.
If $k$ is an algebraically closed field,
then every curve $V$ of genus at most $4$ over $k$ can be modeled by a nondegenerate polynomial having one of the
following as Newton polytope:
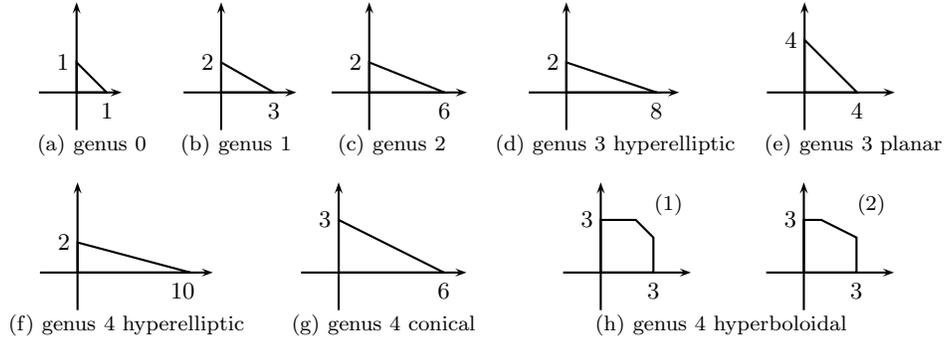
\begin{center}
\begin{pspicture}(-0.5,-0.7)(0.6,1.5)
\pspolygon[fillstyle=solid,linecolor=black](0,0)(0.4,0)(0,0.4)
\psline{->}(-0.5,0)(0.6,0) \psline{->}(0,-0.5)(0,1.2)
\rput(-0.18,0.4){\small $1$} \rput(0.4,-0.25){\small $1$}
\rput(0.2,-0.7){\footnotesize (a) genus $0$}
\end{pspicture}
\qquad
\begin{pspicture}(-0.5,-0.7)(1,1.5)
\pspolygon[fillstyle=solid,linecolor=black](0,0)(0.7,0)(0,0.4)
\psline{->}(-0.5,0)(1,0) \psline{->}(0,-0.5)(0,1.2)
\rput(-0.18,0.4){\small $2$} \rput(0.7,-0.25){\small $3$}
\rput(0.2,-0.7){\footnotesize (b) genus $1$}
\end{pspicture}
\quad
\begin{pspicture}(-0.5,-0.7)(1.3,1.5)
\pspolygon[fillstyle=solid,linecolor=black](0,0)(1,0)(0,0.4)
\psline{->}(-0.5,0)(1.3,0) \psline{->}(0,-0.5)(0,1.2)
\rput(-0.18,0.4){\small $2$} \rput(1,-0.25){\small $6$}
\rput(0.3,-0.7){\footnotesize (c) genus $2$}
\end{pspicture}
\qquad
\begin{pspicture}(-0.5,-0.7)(1.5,1.5)
\pspolygon[fillstyle=solid,linecolor=black](0,0)(1.2,0)(0,0.4)
\psline{->}(-0.5,0)(1.5,0) \psline{->}(0,-0.5)(0,1.2)
\rput(-0.18,0.4){\small $2$} \rput(1.2,-0.25){\small $8$}
\rput(0.65,-0.7){\footnotesize (d) genus $3$ hyperelliptic}
\end{pspicture}
\quad
\begin{pspicture}(-1.2,-0.7)(1.2,1.5)
\pspolygon[fillstyle=solid,linecolor=black](0,0)(0.7,0)(0,0.7)
\psline{->}(-0.5,0)(1.2,0) \psline{->}(0,-0.5)(0,1.2)
\rput(-0.18,0.7){\small $4$} \rput(0.7,-0.25){\small $4$}
\rput(0.65,-0.7){\footnotesize (e) genus $3$ planar}
\end{pspicture}
\end{center}
\vspace{0.3cm}
\begin{center}
\begin{pspicture}(-0.5,-0.7)(1.8,1.2)
\pspolygon[fillstyle=solid,linecolor=black](0,0)(1.5,0)(0,0.4)
\psline{->}(-0.5,0)(1.8,0) \psline{->}(0,-0.5)(0,1.2)
\rput(-0.18,0.4){\small $2$} \rput(1.4,-0.25){\small $10$}
\rput(0.65,-0.7){\footnotesize (f) genus $4$ hyperelliptic}
\end{pspicture}
\qquad \quad
\begin{pspicture}(-0.5,-0.7)(1.7,1.2)
\pspolygon[fillstyle=solid,linecolor=black](0,0)(1.4,0)(0,0.7)
\psline{->}(-0.5,0)(1.7,0) \psline{->}(0,-0.5)(0,1.2)
\rput(-0.18,0.7){\small $3$} \rput(1.4,-0.25){\small $6$}
\rput(0.6,-0.7){\footnotesize (g) genus $4$ conical}
\end{pspicture}
\quad \
\begin{pspicture}(-1.2,-0.7)(3.9,1.2)
\pspolygon[fillstyle=solid,linecolor=black](0,0)(0.7,0)(0.7,0.466)(0.466,0.7)(0,0.7)
\psline{->}(-0.5,0)(1.2,0) \psline{->}(0,-0.5)(0,1.2)
\rput(-0.18,0.7){\small $3$} \rput(0.7,-0.25){\small $3$}
\rput(0.9,0.9){\footnotesize (1)}
\pspolygon[fillstyle=solid,linecolor=black](2.7,0)(3.4,0)(3.4,0.466)(2.933,0.7)(2.7,0.7)
\psline{->}(2.2,0)(3.9,0) \psline{->}(2.7,-0.5)(2.7,1.2)
\rput(2.52,0.7){\small $3$} \rput(3.4,-0.25){\small $3$}
\rput(3.6,0.9){\footnotesize (2)}
\rput(1.6,-0.7){\footnotesize (h) genus $4$ hyperboloidal}
\end{pspicture} \\
\
\end{center}
Moreover, these classes are disjoint.  For the polytopes (c)--(h.1),
we have $\dim \calM_\Delta=3,5,6,7,8,9$, respectively.  All
hyperelliptic curves and $C_{ab}$ curves are nondegenerate.

For an arbitrary perfect field $k$, if $V$ is not hyperboloidal and has
genus at most $4$, then $V$ is nondegenerate whenever $k$ is a
sufficiently large finite field, or when $k$ is infinite and $V(k) \neq
\emptyset$; for the former, the condition $\#k \geq 23$ is
sufficient but most likely not optimal (see Remark~\ref{project}).



\begin{rmk}
We can situate the nonhyperelliptic $C_{ab}$ curves that lie in this
classification. In genus $3$, we have $C_{3,4}$ curves, which have a
smooth model in $\mathbb{P}^2_k$, since $\Delta_{3,4}$ is
nonhyperelliptic. In genus $4$, we have $C_{3,5}$ curves, which are
conical: this can be seen by analyzing the interior lattice points
of $\Delta_{3,5}$, as in Section~\ref{sectionthreefour}.
\end{rmk}

\begin{rmk} \label{project}
In case $\# k < \infty$, we proved (without further condition on $\#k$)
that if $V$ is not hyperboloidal then it can be modeled by a
polynomial $f \in k[x,y]$ with Newton polytope contained in one of
the polytopes (a)--(g). The condition on $\#k$ then came along with
an application of Lemma~\ref{fieldlowerbound} to deduce
nondegeneracy. In the $g=1$ case, we got rid of this condition by
using non-linear transformations (completing the square) and allowing smaller polytopes.
Very likely, similar techniques can be used to improve the bounds on
$\#k$ in genera $2 \leq g \leq 4$. It would be interesting to
investigate this problem more completely and to even produce the
finite list of all curves of genus $\leq 3$ over a finite field of odd characteristic that
are not nondegenerate.
\end{rmk}


\section{An upper bound for $\dim \Mgnd$} \label{SectionBounds}


From now on, we assume $k = \overline{k}$. In this section, we prepare for a proof of Theorem
\ref{reallymaintheorem}, which gives an upper bound for $\dim \Mgnd$ in terms of $g$.

For a lattice polytope $\Delta \subset \mathbb{Z}^2$ with
$g \geq 2$ interior lattice points, we sharpen the obvious upper bound
$\dim \calM_\Delta \leq \dim M_\Delta =\#(\Delta \cap \Z^2)-1$
(see (\ref{dimensionremark})) by incorporating the action of the
automorphism group of $X(\Delta)_k$, which has been explicitly
described by Bruns and Gubeladze \cite[Section 5]{BrunsGubeladze}.
In
Sections~\ref{sectionboundtermsofgenus}--\ref{generalpolytopessection}
we then work towards a bound in terms of $g$, following ideas of
Haase and Schicho \cite{HaaseSchicho}.


The automorphisms of $X(\Delta)_k =
\Proj k[\Delta] \hookrightarrow \PP^{\#(\Delta \cap \Z^2)-1}$ correspond to the
graded $k$-algebra automorphisms of $k[\Delta]$, and admit a
combinatorial description as follows.

\begin{defn}
A nonzero vector $v \in \Z^2$ is a \emph{column vector} of $\Delta$ if there exists a facet $\tau \subset \Delta$ (the \emph{base facet}) such that
\[ v+((\Delta \setminus \tau) \cap \Z^2) \subset \Delta. \]
\end{defn}

We denote by $c(\Delta)$ the number of column vectors of $\Delta$.

\begin{exm}
Any multiple
of the standard $2$-simplex $\Sigma$ has $6$ column vectors. The
octagonal polytope below shows that a polytope may have no column
vectors.

\begin{center} \label{colvecs}
\ \\
\begin{pspicture}(0,0)(10.4,3.4)
\psgrid[gridcolor=lightgray,subgridcolor=lightgray,gridwidth=0.01,subgridwidth=0.01,gridlabels=0,subgriddiv=2]
\pspolygon[fillstyle=none,linecolor=black](0.5,0.5)(4.5,0.5)(4.5,2)(2.5,2.5)
\psline[linecolor=red]{->}(3,1)(3,0.5)
\psline[linecolor=red]{->}(3,1)(2.5,0.5)
\psline[linecolor=red]{->}(2,1.5)(1.5,1.5)
\pspolygon[fillstyle=none,linecolor=black](5,0.5)(7.5,0.5)(5,3)
\psline[linecolor=red]{->}(6,1)(6,0.5)
\psline[linecolor=red]{->}(6,1)(6.5,0.5)
\psline[linecolor=red]{->}(5.5,1.5)(5,1.5)
\psline[linecolor=red]{->}(5.5,1.5)(5,2)
\psline[linecolor=red]{->}(6,1.5)(6.5,1.5)
\psline[linecolor=red]{->}(6,1.5)(6,2)
\pspolygon[fillstyle=none,linecolor=black](8,1)(8.5,0.5)(9.5,0.5)(10,1)(10,2)(9.5,2.5)(8.5,2.5)(8,2)
\end{pspicture} \\
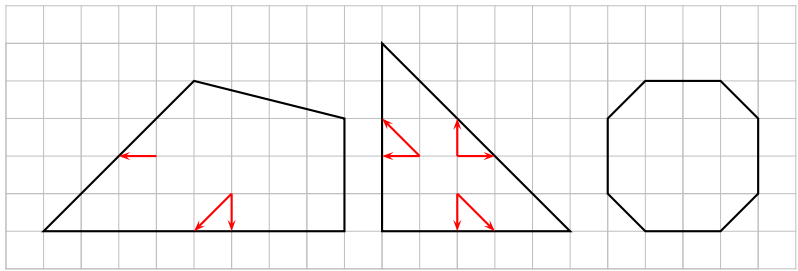
\textbf{Figure \ref{colvecs}}: Column vectors of some lattice polytopes \\
\ \\
\end{center}
\end{exm}

The dimension of the automorphism group $\Aut (X(\Delta)_k)$ is then
determined as follows.

\begin{prop}[Bruns--Gubeladze {\cite[Theorem 5.3.2]{BrunsGubeladze}}] \label{brunsgub}
We have
\[ \dim \Aut (X(\Delta)_k)=c(\Delta)+2. \]
\end{prop}

\begin{proof}[Proof sketch]
One begins with the $2$-dimensional subgroup of $\Aut(X(\Delta)_k)$
induced by the inclusion $\Aut(\T^2) \hookrightarrow
\Aut(X(\Delta)_k)$. On the $k[\Delta]$-side, this corresponds to the
graded automorphisms induced by $(x,y) \mapsto (\lambda x, \mu y)$ for $\lambda, \mu \in k^{\ast 2}$.

Next, column
vectors of $\Delta$ correspond to automorphisms of $X(\Delta)_k$ in
the following way. If $v$ is a column vector, modulo equivalence
we may assume that $v=(0,-1)$, that the
base facet is supported on the $X$-axis, and that $\Delta$ is
contained in the positive quadrant $\R_{\geq 0}^2$.  Let $f(x,y) \in
k[x,y]$ be supported on $\Delta$. Since the vector $v=(0,-1)$ is a
column vector, the polynomial $f(x,y + \lambda)$ will again be
supported on $\Delta$, for any $\lambda \in k$. Hence $v$ induces a
family of graded automorphisms $k[\Delta] \to k[\Delta]$,
corresponding to a one-dimensional subgroup of $\Aut (X(\Delta)_k)$.


It then remains to show that these subgroups are algebraically
independent from each other and from $\Aut(\T^2)$, and that together
they generate $\Aut(X(\Delta)_k)$ (after including the finitely many
automorphisms coming from $\mathbb{Z}$-affine transformations mapping
$\Delta$ to itself).
\end{proof}


Using the fact that a curve of genus $g
\geq 2$ has finitely many automorphisms we obtain the following
corollary. We leave the details as an exercise.

\begin{cor} \label{boundcolumn}
We have $\dim \calM_\Delta \leq m(\Delta):=\#(\Delta \cap
\Z^2)-c(\Delta)-3$.
\end{cor}

\begin{exa} \label{boundsharp}
Let $\Delta=\conv(\{(0,0),(2g+2,0),(0,2)\})$ as in
section~\ref{sectionhyperelliptic}, so that $\dim \mathcal{M}_\Delta =
2g - 1$. One verifies that $c(\Delta)=g + 3$, so the upper
bound in Corollary~\ref{boundcolumn} reads
$m(\Delta)=(3g+6)-(g+3)-3=2g$ which is indeed bigger than $2g-1$. In the case
$\Delta=\conv(\{(0,0),(2g+1,0),(0,2)\})$, corresponding to
Weierstrass models having a unique point at infinity, one exactly
finds $m(\Delta) = 2g-1$. So in this case, the bound is sharp. It is
easy to verify that the bound is also sharp if $\Delta=d\Sigma$
$d \in \mathbb{Z}_{\geq 4}$; then $\dim
\calM_\Delta$ reads $(d+1)(d+2)/2-9=g+3d-9 \leq 2g$.
\end{exa}

\begin{rmk}
Temporarily looking ahead to Section~\ref{morecurves}, we note that
the lattice polytopes $d\Sigma$ ($d \in \mathbb{Z}_{\geq 4}$) are
examples of so-called \emph{maximal polytopes} (see
Section~\ref{maximalpolytopessection} for the definition). For this class of polytopes,
Corollary~\ref{boundcolumn} will always give a sharp estimate. This
is the main result of Koelman's thesis \cite[Theorem
2.5.12]{Koelman}.
\end{rmk}

\begin{exa} \label{brunsexamplecab}
We now use Corollary~\ref{boundcolumn} to show that the dimension of
the moduli space of nonhyperelliptic $C_{ab}$ curves of genus $g$ (where $a$ and $b$ may vary)
has dimension strictly smaller than $2g-1=\dim
\calH_g$ whenever $g \geq 6$.  Consider $\Delta_{ab} =
\text{Conv}\{(0,a),(b,0),(0,0)\}$ with $a,b \in \Z_{\geq 3}$
coprime. Then we have
\[ g = (a-1)(b-1)/2, \quad \# (\Delta \cap \mathbb{Z}^2) = g + a + b + 1, \]
and the set of column vectors is given by
\[ \{ (n, -1) : n = 0,\dots,\lfloor b/a \rfloor\} \cup \{(-1, m) : m = 0, \dots,\lfloor a/b \rfloor\}. \]
Suppose without loss of generality that $a < b$. Then $a$ is bounded
by $\sqrt{2g} + 1$. Corollary~\ref{boundcolumn} yields
\[ \dim \calM_\Delta \leq m(\Delta)= g + a + b + 1 - \left(\left\lfloor\frac{b}{a}\right\rfloor+2\right) - 3
< a + \frac{2g-1}{a} + g - 2. \] As a (real) function of $a$, this upper bound has a unique minimum at
$a=\sqrt{2g-1}$. Therefore, to deduce that it is strictly smaller
than $2g-1$ for all $a \in [3,\sqrt{2g} + 1]$, it suffices to verify
so for the boundary values $a=3$ and $a=\sqrt{2g}+1$, which is indeed
the case if $g \geq 6$.
\end{exa}

\section{A bound in terms of the genus}
\label{sectionboundtermsofgenus}

Throughout the rest of this article, we will employ the following
notation. Let $\Delta^{(1)}$ be the convex hull of the interior
lattice points of $\Delta$.  Let $r$ (resp.\ $r^{(1)}$) denote the
number of lattice points on the boundary of $\Delta$ (resp.\
$\Delta^{(1)}$), and let $g^{(1)}$ denote the number of interior
lattice points in $\Delta^{(1)}$, so that $g=g^{(1)}+r^{(1)}$.

We now prove the following preliminary bound.

\begin{prop} \label{maintheorem}
If $\Delta$ has at least $g \geq 2$ interior lattice points, then
$\dim \calM_\Delta \leq 2g+3$.
\end{prop}

\begin{proof}
We may assume that $\Delta$ is nonhyperelliptic, because otherwise
$\dim \mathcal{M}_\Delta \leq 2g-1$ by Lemma~\ref{hyperelliptic}. We
may also assume that $\Delta^{(1)}$ is not a multiple of $\Sigma$,
since otherwise $\Delta$-nondegenerate curves are canonically
embedded in $X(\Delta^{(1)})_k \cong \mathbb{P}^2_k$ using
Proposition~\ref{internallatticegenus}; then from
Example~\ref{boundsharp} it follows that $\dim \calM_\Delta \leq
2g$.

An upper bound for $\dim \mathcal{M}_\Delta$ in terms of $g$ then
follows from a lemma by Haase and Schicho \cite[Lemma
12]{HaaseSchicho}, who proved that $r \leq r^{(1)} + 9$, in which
equality holds if and only if $\Delta=d\Sigma$ for some $d \in
\Z_{\geq 4}$ (a case which we have excluded). Hence
\begin{equation} \label{boundHS}
\# (\Delta \cap \mathbb{Z}^2) = g + r \leq g + r^{(1)} + 8 = 2g + 8 - g^{(1)},
\end{equation}
and thus
\begin{equation} \label{boundwithoutc}
 \dim \mathcal{M}_\Delta \leq m(\Delta)=\#(\Delta \cap
\mathbb{Z}^2) - c(\Delta) - 3 \leq 2g + 5 - c(\Delta) - g^{(1)}
\leq 2g+5.
\end{equation}
This bound improves to $2g+3$ if $g^{(1)} \geq 2$, so we remain with
two cases: $g^{(1)} = 0$ and $g^{(1)} = 1$.

 Suppose first that $g^{(1)}=0$. Then by
Lemma~\ref{nointeriorimpliestrigonal} below, any
$\Delta$-nondegenerate curve is either a smooth plane quintic
(excluded), or a trigonal curve. Since the moduli space of trigonal
curves has dimension $2g+1$ (a classical result, see also
Section~\ref{morecurves} below), the bound holds.

Next, suppose that $g^{(1)}=1$. Then, up to equivalence, there are
only $16$ possibilities for $\Delta^{(1)}$, which are listed in
\cite[Figure 2]{PoonenRodriguez} or in the appendix below. Hence,
there are only finitely many possibilities for $\Delta$, and for
each of these polytopes we find that $\#(\Delta \cap
\Z^2)-c(\Delta)-3 \leq 2g+2$.
\end{proof}

In fact, for all but the $5$ polytopes in Figure \ref{fig2gplus2}
(up to equivalence), we find that the stronger bound $\#(\Delta
\cap \Z^2)-c(\Delta)-3 \leq 2g+1$ holds.

\begin{center} \label{fig2gplus2}
\ \\
\begin{pspicture}(-0.5,-0.5)(12,3)
\psgrid[gridcolor=lightgray,subgridcolor=lightgray,gridwidth=0.01,subgridwidth=0.01,gridlabels=0,subgriddiv=2]
\pspolygon[fillstyle=none,linecolor=black,linestyle=solid](-0.5,1.5)(-0.5,2.5)(0.5,2.5)(1.5,1.5)(1.5,0.5)(0.5,0.5)
\pspolygon[fillstyle=none,linecolor=black,linestyle=dashed](0,1.5)(0,2)(0.5,2)(1,1.5)(1,1)(0.5,1)
\rput(0.5,0){(a)}
\pspolygon[fillstyle=none,linecolor=black,linestyle=solid](2,0.5)(2,1.5)(3,2.5)(4,2.5)(4,1)(3.5,0.5)
\pspolygon[fillstyle=none,linecolor=black,linestyle=dashed](2.5,1)(2.5,1.5)(3,2)(3.5,2)(3.5,1)
\rput(3,0){(b)}
\pspolygon[fillstyle=none,linecolor=black,linestyle=solid](5,0.5)(6.5,0.5)(6.5,2)(6,2.5)(4.5,2.5)(4.5,1)
\pspolygon[fillstyle=none,linecolor=black,linestyle=dashed](5,1)(6,1)(6,2)(5,2)
\rput(5.5,0){(c)}
\pspolygon[fillstyle=none,linecolor=black,linestyle=solid](7,1)(7,2.5)(8,2.5)(9.5,1)(9.5,0.5)(7.5,0.5)
\pspolygon[fillstyle=none,linecolor=black,linestyle=dashed](7.5,1)(7.5,2)(8,2)(9,1)
\rput(8.25,0){(d)}
\pspolygon[fillstyle=none,linecolor=black,linestyle=solid](12,0.5)(12,2.5)(11.5,3)(9.5,3)(9.5,2.5)(11.5,0.5)
\pspolygon[fillstyle=none,linecolor=black,linestyle=dashed](11.5,1)(11.5,2.5)(10,2.5)
\rput(10.75,0){(e)}

\end{pspicture} \\
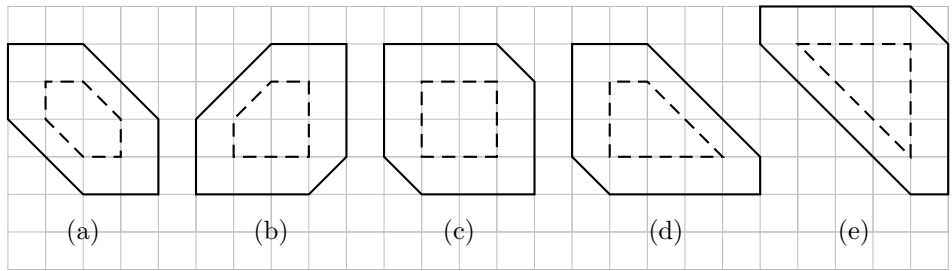
\textbf{Figure \ref{fig2gplus2}}: Polytopes with $g^{(1)}=1$ and $\#(\Delta \cap \Z^2)-c(\Delta)-3 = 2g+2$ \\
\ \\
\end{center}

\begin{lem} \label{nointeriorimpliestrigonal}
If $\Delta^{(1)}$ is a $2$-dimensional polytope having no interior
lattice points, then any $\Delta$-nondegenerate curve is either
trigonal, either isomorphic to a smooth quintic in $\mathbb{P}^2_k$.
\end{lem}

\begin{proof}
Koelman gives a proof of this in his Ph.D. thesis \cite[Lemma
3.2.13]{Koelman}, based on Petri's theorem. A more combinatorial
argument uses the fact that lattice polytopes of genus 0 are
equivalent with either $2\Sigma$, or with a polytope that is caught
between two horizontal lines of distance $1$. This was proved
independently by Arkinstall, Khovanskii, Koelman, and Schicho (see
the generalized statement by Batyrev-Nill \cite[Theorem
2.5]{BatyrevNill}).

In the first case, $\Delta$-nondegenerate curves are canonically
embedded in $X(2\Sigma)_k \cong \mathbb{P}^2_k$, hence they are
isomorphic to smooth plane quintics.

In the second case, it follows that $\Delta$ is caught between two
horizontal lines of distance $3$. This may actually fail if $\Delta^{(1)} =
\Sigma$, which corresponds to smooth plane quartics. But in both
situations, $\Delta$-nondegenerate curves are trigonal.
\end{proof}

\section{Refining the upper bound: Maximal polytopes} \label{maximalpolytopessection}

We further refine the bound in Proposition \ref{maintheorem} by
adapting the proof of the Haase--Schicho bound $r \leq r^{(1)} + 9$
in order to obtain an estimate for $r - c(\Delta)$ instead of just
$r$. We first do this for \emph{maximal} polytopes, and treat
nonmaximal polytopes in the next section.

\begin{defn} \label{maximaldef}
A lattice polytope $\Delta \subset \Z^n$ is \emph{maximal} if $\Delta$ is not properly contained in another lattice polytope with the same interior lattice points, i.e., for all lattice polytopes $\Delta' \supsetneq \Delta$, we have
\[ \inter(\Delta') \cap \Z^n \neq \inter(\Delta) \cap \Z^n. \]
\end{defn}

We define the \emph{relaxed polytope} $\Delta^{(-1)}$ of a lattice
polytope $\Delta \subset \Z^2$ as follows.  Assume that $0 \in
\Delta$.  To each facet $\tau \subset \Delta$ given by an inequality
of the form $a_1X + a_2Y \leq b$ with $a_i \in \Z$ coprime, we
define the \emph{relaxed inequality} $a_1X + a_2Y \leq b+1$ and let
$\Delta^{(-1)}$ be the intersection of these relaxed inequalities.
If $p$ is a vertex of $\Delta$ given by the intersection of two such
facets, we define the \emph{relaxed vertex} $p^{(-1)}$ to be the
intersection of the boundaries of the corresponding relaxed inequalities.


\begin{lem}[Haase--Schicho {\cite[Lemmas 9--10]{HaaseSchicho}}, Koelman {\cite[Section 2.2]{Koelman}}] \label{maximalitycriterion}
Let $\Delta \subset \Z^2$ be a $2$-dimensional lattice polytope.
Then $\Delta^{(-1)}$ is a lattice polytope if and only if $\Delta =
\Delta'^{(1)}$ for some lattice polytope $\Delta'$. Furthermore, if
$\Delta$ is nonhyperelliptic, then $\Delta$ is maximal if and only
if $\Delta=(\Delta^{(1)})^{(-1)}$.
\end{lem}


The proof of the Haase--Schicho bound $r \leq r^{(1)} + 9$ utilizes
a theorem of Poonen and Rodriguez-Villegas \cite{PoonenRodriguez},
which we now introduce.

A \emph{legal move} is a pair $(v,w)$ with $v,w \in \Z^2$ such that
$\conv(\{0,v,w\})$ is a $2$-dimensional triangle whose only nonzero
lattice points lie on $e(v,w)$, the edge between $v$ and $w$. The
\emph{length} of a legal move $(v,w)$ is
\[ \ell(v,w)= \det \begin{pmatrix} v \\ w \end{pmatrix}, \]
which is of absolute value $r-1$, where $r=\#(e(v,w) \cap \Z^2)$
is the number of lattice points on the edge between $v$ and $w$.
Note that the length can be negative.

A \emph{legal loop} $\calP$ is a sequence of vectors $v_1, v_2,
\dots, v_n \in \mathbb{Z}^2$ such that for all $i=1, \dots, n$ and indices taken modulo $n$, we have:
\begin{itemize}
\item $(v_i, v_{i+1})$ is a legal move, and
\item $v_{i-1}, v_i, v_{i+1}$ are not contained in a line.
\end{itemize}
The \emph{length} $\ell(\calP)$ of a legal loop $\calP$ is the sum of the lengths of
its legal moves.

The \emph{winding number} of a legal loop is its winding number
around $0$ in the sense of algebraic topology. The \emph{dual
loop} $\calP\spcheck$ is given by $w_1, \dots, w_n$, where $w_i =
\ell(v_i,v_{i+1})^{-1} \cdot (v_{i+1}-v_i)$ for $i=1, \dots, n$. One can
check that this is again a legal loop with the same winding number
as $\calP$ and that $\calP\spcheck{}\spcheck = \calP$ after a
$180^\circ$ rotation.

\begin{thm}[Poonen--Rodriguez-Villegas {\cite[Section 9.1]{PoonenRodriguez}}] \label{Poonen} Let
$\calP$ be a legal loop with winding number $w$. Then
$\ell(\calP) + \ell(\calP\spcheck) = 12w$.\end{thm}

Now let $\Delta \subset \Z^2$ be a maximal polytope with
$2$-dimensional interior $\Delta^{(1)}$.  We associate to $\Delta$
a legal loop $\calP(\Delta)$ as follows. By Lemma
\ref{maximalitycriterion}, $\Delta$ is obtained from
$\Delta^{(1)}$ by relaxing the edges.  Let $p_1,\dots,p_n$ be the
vertices of $\Delta^{(1)}$, enumerated counterclockwise; then
$\calP(\Delta)$ is given by the sequence $q_i=p_i^{(-1)}-p_i$
where $p_i^{(-1)}$ is the relaxed vertex of $p_i$.

\begin{exm} \label{legalloop}
The following picture, inspired by Haase--Schicho \cite[Figure 20]{HaaseSchicho},
is illustrative: it shows a polytope $\Delta$ with $2$-dimensional
interior $\Delta^{(1)}$, the associated legal loop $\calP(\Delta)$,
and its dual $\calP(\Delta)\spcheck$.  In this example,
$\ell(\calP(\Delta)) = \ell(\calP(\Delta)\spcheck) = 6$.

\begin{center}
\ \\
\begin{pspicture}(0,0.51)(12,3.4)
\psgrid[gridcolor=lightgray,subgridcolor=lightgray,gridwidth=0.01,subgridwidth=0.01,gridlabels=0,subgriddiv=2]
\pspolygon[fillstyle=none,linecolor=black](0.5,1)(4.5,1)(4.5,2)(2.5,3)
\pspolygon[fillstyle=none,linestyle=dashed,linecolor=black](1.5,1.5)(4,1.5)(4,2)(3,2.5)(2.5,2.5)
\psline[linecolor=blue]{->}(4,1.5)(4.5,1)
\psline[linecolor=blue]{->}(4,2)(4.5,2)
\psline[linecolor=blue]{->}(3,2.5)(2.5,3)
\psline[linecolor=blue]{->}(2.5,2.5)(2.5,3)
\psline[linecolor=blue]{->}(1.5,1.5)(0.5,1)

\psline[linecolor=blue]{->}(7.5,2)(8,1.5)
\psline[linecolor=blue]{->}(7.5,2)(8,2)
\psline[linecolor=blue]{->}(7.5,2)(7,2.5)
\psline[linecolor=blue]{->}(7.5,2)(7.5,2.5)
\psline[linecolor=blue]{->}(7.5,2)(6.5,1.5)
\psline{->}(8,1.5)(8,2)
\psline{->}(8,2)(7,2.5)
\psline{->}(7,2.5)(7.5,2.5)
\psline{->}(7.5,2.5)(6.5,1.5)
\psline{->}(6.5,1.5)(8,1.5)

\pscircle[fillstyle=solid,fillcolor=black](7.5,1.5){0.05}
\pscircle[fillstyle=solid,fillcolor=black](7,1.5){0.05}
\pscircle[fillstyle=solid,fillcolor=black](7,2){0.05}

\psline[linecolor=blue]{->}(10.5,2.5)(9.5,2.5)
\psline[linecolor=blue]{->}(9.5,2.5)(10,2)
\psline[linecolor=blue]{->}(10,2)(10,1.5)
\psline[linecolor=blue]{->}(10,1.5)(11,2)
\psline[linecolor=blue]{->}(11,2)(10.5,2.5)
\psline{->}(10.5,2)(10.5,2.5)
\psline{->}(10.5,2)(9.5,2.5)
\psline{->}(10.5,2)(11,2)
\psline{->}(10.5,2)(10,2)
\psline{->}(10.5,2)(10,1.5)
\end{pspicture} \\
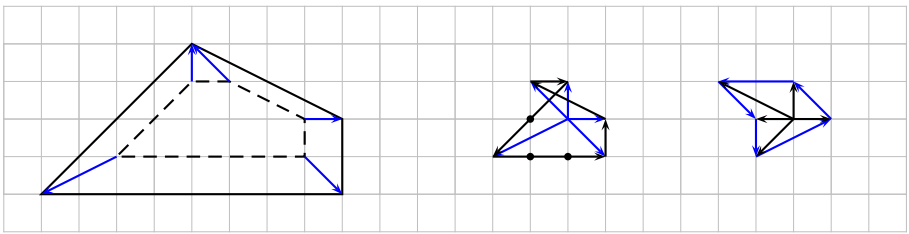
\textbf{Figure \ref{legalloop}}: The legal loop $\calP(\Delta)$ associated to a lattice polytope $\Delta$ \\
\end{center}
\end{exm}

A crucial observation is that the bold-marked lattice points of
$\calP(\Delta)$ are column vectors of $\Delta$.  This holds in
general and lies at the core of our following refinement of the
Haase--Schicho bound.

\begin{lem} \label{boundbminc}
If $\Delta$ is maximal and nonhyperelliptic, then:
\begin{enumalph}
\item $r - r^{(1)} = \ell(\calP(\Delta)) \leq 9$.
\item $r - r^{(1)} - c(\Delta) \leq \min\left(\ell(\calP(\Delta)),\ell(\calP(\Delta)\spcheck)\right) \leq 6$.
\end{enumalph}
\end{lem}

\begin{proof}
We abbreviate $\calP=\calP(\Delta)$.

Inequality (a) is by Haase--Schicho \cite[Lemma 11]{HaaseSchicho}
and works as follows. The length of the legal move $(q_i,q_{i+1})$
measures the difference between the number of lattice points on the
facet of $\Delta$ connecting $p_i^{(-1)}$ and $p_{i+1}^{(-1)}$, and
the number of lattice points on the edge of $\Delta^{(1)}$
connecting $p_i$ and $p_{i+1}$. Therefore $r - r^{(1)} =
\ell(\calP)$.  The dual loop $\calP\spcheck$ walks (in a consistent
and counterclockwise-oriented way) through the normal vectors of
$\Delta^{(1)}$, therefore each move has positive length and we have
$\ell(\calP(\Delta)\spcheck)\geq 3$. Since $\calP\spcheck$ has
winding number $1$, the statement follows from Theorem~\ref{Poonen}.
(One can further show that equality holds if and only if $\Delta$ is
a multiple of the standard $2$-simplex $\Sigma$.)

To prove inequality (b), we first claim: there is a bijection
between lattice points $v$ which lie properly on a
counterclockwise-oriented (positive length) legal move $q_iq_{i+1}$
of $\calP$, and column vectors of $\Delta$ with base facet
$p_i^{(-1)}p_{i+1}^{(-1)}$. Indeed, after an appropriate
transformation, we may assume as in Proposition \ref{brunsgub} that
$v=(0,-1)$, that $p_i^{(-1)}$ and $p_{i+1}^{(-1)}$ lie on the
$X$-axis, and that $\Delta$ is contained in the positive quadrant
$\R_{\geq 0}^2$; after these normalizations, the claim is straightforward.

Now, since the dual loop $\calP\spcheck$ consists of
counterclockwise-oriented legal moves only, it has at most
$\ell(\calP\spcheck)$ vertices. Since $\calP =
\calP\spcheck{}\spcheck$ (after $180^\circ$ rotation), $\calP$ has
at most $\ell(\calP\spcheck)$ vertices.  By the claim, we have
$\ell(\calP) \leq \ell(\calP\spcheck) + c$, and the result follows
by combining this with part (a) and Theorem~\ref{Poonen}.
\end{proof}

\begin{cor} \label{max2gplus1}
If $\Delta$ is maximal, then $\dim \mathcal{M}_\Delta \leq 2g + 3 -
g^{(1)}$. In particular, if $g^{(1)} \geq 2$ then $\dim
\mathcal{M}_\Delta \leq 2g+1$.
\end{cor}

\begin{proof}
By Lemma \ref{boundbminc}, we have $ m(\Delta)=g+r-3-c(\Delta)
\leq g+r^{(1)}+3 \leq 2g+3-g^{(1)}$.
\end{proof}

\begin{rmk}
Note that Lemma~\ref{boundbminc}(a) immediately extends to
nonmaximal polytopes ($r - r^{(1)}$ can only decrease), so the
Haase--Schicho bound holds for arbitrary nonhyperelliptic polytopes.
This we cannot conclude for part (b): if $r$ decreases, $c(\Delta)$
may decrease more quickly so that the bound no longer holds. An
example of such behaviour can be found in
Figure~\ref{fig2gplus2}(c).
\end{rmk}

\section{Refining the upper bound: general polytopes}
\label{generalpolytopessection}

We are now ready to prove the main result of
Sections~\ref{SectionBounds}--\ref{generalpolytopessection}.
\begin{thm} \label{reallymaintheorem}
If $g \geq 2$, then $\dim \Mgnd \leq 2g+1$ except for $g=7$ where we have $\dim \calM_7^{\textup{nd}} \leq 16$.
\end{thm}

\begin{proof}
It suffices to show that the claimed bounds hold for all polytopes
$\Delta$ with $g$ interior lattice points. By the proof of
Proposition~\ref{maintheorem}, we may assume that $\Delta^{(1)}$ is
two-dimensional, that it is not a multiple of $\Sigma$, and that it
has $g^{(1)} \geq 1$ interior lattice points.

Let us first assume that $g^{(1)} \geq 2$. We will show that $\dim
\calM_\Delta \leq 2g+ 1$. From Corollary~\ref{max2gplus1}, we know
that this is true if $\Delta$ is maximal. Therefore, suppose that
$\Delta$ is nonmaximal; then it is obtained from a maximal polytope
$\widetilde{\Delta}$ by taking away points on the boundary (keeping
the interior lattice points intact). If two or more boundary points
are taken away, then as in (\ref{boundwithoutc}) we have
\[ m(\Delta) \leq \#(\Delta \cap \mathbb{Z}^2) - 3 \leq \#(\widetilde{\Delta} \cap \mathbb{Z}^2) - 2 - 3 \leq 2g + 5 -
g^{(1)} - 2 \leq 2g+1. \] So we may assume that
$\Delta=\conv(\widetilde{\Delta} \cap \Z^2 \setminus \{p\})$ for a
vertex $p \in \widetilde{\Delta}$. Similarly, we may assume that
$c(\Delta) < c(\widetilde{\Delta})$, for else
\[ m(\Delta) = \#(\Delta \cap \mathbb{Z}^2) - c(\Delta) - 3 \leq \#(\widetilde{\Delta} \cap \mathbb{Z}^2) - c(\widetilde{\Delta}) -
3 = m(\widetilde{\Delta}) \leq 2g + 1. \] Let $v$ be a column vector
of $\widetilde{\Delta}$ that is no longer a column vector of
$\Delta=\conv(\widetilde{\Delta} \cap \Z^2 \setminus \{p\})$. Then
$p$ must lie on the base facet $\tau$ of $v$. After an appropriate
transformation, we may assume that $p = (0,0)$, that
$v=(0,-1)$, that $\tau$ lies along the $X$-axis, and that
$\widetilde{\Delta}$ lies in the positive quadrant, as follows.



\begin{center} \label{nonmaximal}
\begin{pspicture}(0,1)(5,5)
\psgrid[gridcolor=lightgray,subgridcolor=lightgray,gridwidth=0.01,subgridwidth=0.01,gridlabels=0,subgriddiv=2]
\psline{->}(0,2)(5,2) \psline{->}(1,1)(1,5)
\pspolygon[fillstyle=none,linecolor=black](1,2)(1,3)(3,4)(4,2)
\psline[linecolor=red]{->}(2.5,2.5)(2.5,2)
\psline[linestyle=dashed](1.5,2)(1,2.5) \rput(0.8,1.8){\small $p$}
\pscircle[fillstyle=solid,fillcolor=black](1,2){0.05}
\rput(2.7,2.25){\small \textcolor{red}{$v$}}
\end{pspicture}\\
\textbf{Figure \ref{nonmaximal}.2}: An almost maximal polytope.\\
\end{center}

Note that $(1,1) \in \inter(\widetilde{\Delta})$ since otherwise $v$
would still be a column vector of $\Delta$. But then the other facet
of $\widetilde{\Delta}$ which contains $p$ must be supported on the
$Y$-axis, for else $(1,1)$ would no longer be in $\inter(\Delta)$.
One can now verify that if $f(x,y)$ is $\Delta$-nondegenerate, then
for all but finitely many $\lambda \in k$, the polynomial $f(x,y +
\lambda)$ will have Newton
polytope $\widetilde{\Delta}$ and 
all but finitely of those will be
$\widetilde{\Delta}$-nondegenerate. Therefore, we have $\calM_\Delta
\subset \calM_{\widetilde{\Delta}}$, and the dimension estimate
follows.

Now suppose that $g^{(1)} = 1$. From the finite computation in the
proof of Proposition~\ref{maintheorem}, we know that the bound $\dim
\calM_\Delta \leq 2g+1$ holds if $\Delta$ is not among the polytopes
listed in Figure~\ref{fig2gplus2}. Now in this list, the polytopes (b)--(e) are
not maximal, and for these polytopes the same trick as in the $g^{(1)}
\geq 2$ case applies. However, polytope (a) is maximal and contains
7 interior lattice points: therefore, we can only prove $\dim
\mathcal{M}^\textup{nd}_7 \leq 16$.
\end{proof}

Let $\Delta$ be a nonmaximal nonhyperelliptic lattice polytope, and
let $\widetilde{\Delta} = (\Delta^{(1)})^{(-1)}$ be the smallest
maximal polytope containing $\Delta$. Let $f \in k[x^\pm, y^\pm]$ be
a $\Delta$-nondegenerate Laurent polynomial. Since $\Delta \subset
\widetilde{\Delta}$, we can consider the (degree $1$) locus
$\widetilde{V}$ of $f=0$ in $X(\widetilde{\Delta})_k = \Proj \,
k[\widetilde{\Delta}]$.
Then one can wonder whether the observation we made in the proof of
Theorem~\ref{reallymaintheorem} holds in general: is there always a
$\sigma \in \Aut(X(\widetilde{\Delta})_k)$ such that
$\sigma(\widetilde{V}) \cap \T^2_k$ is defined by a
$\widetilde{\Delta}$-nondegenerate polynomial? The answer is no,
because it is easy to construct examples where the only
automorphisms of $X(\widetilde{\Delta})_k$ are those induced by
$\Aut (\T^2_k)$. Then $\sigma(\widetilde{V}) \cap \T^2_k$ is always
defined by $f(\lambda x, \mu y)$ (for some $\lambda, \mu \in k^*$),
which does not have $\widetilde{\Delta}$ as its Newton polytope and hence
cannot be $\widetilde{\Delta}$-nondegenerate.

However, $f$ is very close to being
$\widetilde{\Delta}$-nondegenerate, and this line of thinking leads
to the following observation. Let $p$ be a vertex of
$\widetilde{\Delta}$ that is not a vertex of $\Delta$, and let
$q_1,q_2$ be the closest lattice points to $p$ on the respective
facets of $\widetilde{\Delta}$ containing $p$. The triangle spanned
by $p,q_1,q_2$ cannot contain any other lattice points, because
otherwise removing $p$ would affect the interior of
$\widetilde{\Delta}$. Thus the volume of this triangle is equal to
$1/2$ by Pick's theorem, and the affine chart of
$X(\widetilde{\Delta})_k$ attached to the cone at $p$ is isomorphic
to $\mathbb{A}^2_k$. In particular, $X(\widetilde{\Delta})_k$ is
nonsingular in the zero-dimensional torus $\T_p$ corresponding to
$p$. Then $f$ fails to be $\widetilde{\Delta}$-nondegenerate only
because $\widetilde{V}$ passes through $\T_p$ (i.e.\ passes through
$(0,0) \in \A^2_k$), elsewhere it fulfils the conditions of
nondegeneracy: $\widetilde{V}$ is smooth, intersects the
$1$-dimensional tori associated to the facets of
$\widetilde{\Delta}$ transversally, and does not contain the
singular points of $X(\widetilde{\Delta})_k$.
Now following the methods of Section~\ref{sectionmoduliformulation},
one could construct the bigger moduli space of curves satisfying
this weaker nondegeneracy condition. Its dimension would still be
bounded by $\#( \widetilde{\Delta} \cap \mathbb{Z}^2) -
c(\widetilde{\Delta}) - 3$, which by Lemma~\ref{boundbminc} is at
most $2g+3 - g^{(1)}$ because $\widetilde{\Delta}$ is maximal.
Therefore $\dim \calM_\Delta \leq 2g+3 - g^{(1)}$ for nonmaximal
$\Delta$, and this yields an alternative proof of
Theorem~\ref{reallymaintheorem}. Related observations have been made
by Koelman \cite[Section 2.6]{Koelman}.

\section{Koelman's theorem and a lower bound for $\dim \Mgnd$} \label{morecurves}

For $g \geq 2$, we implicitly proved in Section~\ref{sectionhyperelliptic} that
$\dim \Mgnd \geq 2g - 1$. But already in genera $3$
and $4$, by the results in Section~\ref{sectionthreefour} we have
$\dim \calM_{3}^{\textup{nd}}=6$ and $\dim
\calM_{4}^{\textup{nd}}=9$, so this lower bound is an
underestimation. For higher genera, we prove in this last section
that the bounds given in Theorem~\ref{reallymaintheorem} are sharp.
This follows from the following main result of Koelman's Ph.D.
thesis.

\begin{thm}[{Koelman \cite[Theorem 2.5.12]{Koelman}}] \label{Koelmanmain}
Let $\Delta \subset \R^2$ be a maximal nonhyperelliptic lattice
polytope. Then
\[ \dim \calM_\Delta = \#(\Delta \cap \mathbb{Z}^2) - 1 - \dim \Aut(X(\Delta)_k). \]
\end{thm}

In fact, Koelman assumes $k = \mathbb{C}$ and works with a slightly
bigger moduli space in which ours is dense. But his methods extend
to an arbitrary algebraically closed field $k=\kbar$. Our main result is
then the following.

\begin{thm} \label{lowerbound}
If $g \geq 4$, then $\dim \Mgnd = 2g+1$ except for $g \neq 7$ where
$\dim \calM_7^{\textup{nd}} = 16$.
\end{thm}

\begin{proof} It suffices to find for every genus $g
\geq 5$ a lattice polytope $\Delta$ with $g$ interior lattice
points, for which $\dim \mathcal{M}_\Delta = 2g + 1$ if $g \neq 7$,
and $\dim \mathcal{M}_\Delta = 16$ if $g = 7$. If $g = 2h$ is even,
let $\Delta$ be the rectangle
\begin{equation} \label{trigonaleven}
\conv\left(\{ (0,0),(0,3),(h+1,3),(h+1,0) \} \right). \end{equation}
Note that then $\#(\Delta \cap \Z^2) = 2g + 8$ and $c(\Delta) = 4$.
If $g=2h+1$ is odd but different from $7$, let $\Delta$ be the
trapezium \begin{equation} \label{trigonalodd}
\conv\left(\{(0,0),(0,3),(h,3),(h+3,0) \} \right). \end{equation}
Again, $\#(\Delta \cap \Z^2) = 2g + 8$ and $c(\Delta) = 4$. Finally,
if $g=7$ then let $\Delta$ be
\[ \conv \{(2,0),(0,2),(-2,2),(-2,0),(0,-2),(2,-2)\} \]
(i.e.\ the polytope given in Figure~\ref{fig2gplus2}(a)). Here,
$\#(\Delta \cap \Z^2) = 19$ and $c(\Delta) = 0$. In every case,
$\Delta$ is maximal and the result follows from Koelman's theorem,
when combined with Proposition~\ref{brunsgub}.
\end{proof}

\subsection*{Trigonal curves}
For a class of polytopes including (\ref{trigonaleven}) and
(\ref{trigonalodd}), Koelman's theorem can be proven in a more
elementary way, based on the well-known theory of trigonal curves
\cite{Coppens,Maroni}. For any $k, \ell \in \Z_{\geq 2}$ with $k \leq
\ell$, let $\Delta^{(1)}$ be the trapezium
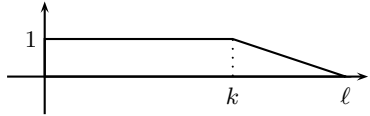
\begin{center}
\begin{pspicture}(-0.5,-0.5)(4.5,1)
\pspolygon[fillstyle=solid,linecolor=black](0,0)(4,0)(2.5,0.5)(0,0.5)
\psline{->}(-0.5,0)(4.3,0) \psline{->}(0,-0.5)(0,1)
\psline[linestyle=dotted]{-}(2.5,0)(2.5,0.5)
\rput(-0.18,0.5){\small $1$} \rput(2.5,-0.25){\small $k$}
\rput(4,-0.25){\small $\ell$}
\end{pspicture}
\end{center}
and let $\Delta = \Delta^{(1)(-1)}$. Then if a curve $V$ is
$\Delta$-nondegenerate, it is trigonal of genus $g= k+\ell+2$. By
Proposition~\ref{internallatticegenus}, it can be canonically
embedded in $X(\Delta^{(1)})_k$, which is the rational surface
scroll $S_{k,\ell} \subset \mathbb{P}^{g-1}_k$. By Petri's theorem
\cite{ACGH}, this scroll is the intersection of all quadrics
containing the canonical embedding. As a consequence, two different
such canonical embeddings must differ by an automorphism of
$\Aut(\PP^{g-1}_k)$ that maps $X(\Delta^{(1)})_k$ to itself; in
other words, any two canonical embeddings of $V$ must differ by an
automorphism of $X(\Delta^{(1)})_k$.

Now let $f_1,f_2 \in k[x^\pm,y^\pm]$ be $\Delta$-nondegenerate
polynomials such that $V(f_1)$ and $V(f_2)$ are isomorphic as
abstract curves. Since the fans associated to $\Delta$ and
$\Delta^{(1)}$ are the same, we have $X(\Delta)_ k =
X(\Delta^{(1)})_k$. Under this identification, $V(f_1)$ and $V(f_2)$
become canonical curves that must differ by an automorphism of
$X(\Delta)_k$. This proves Koelman's theorem for this particular
class of polytopes, which suffices to deduce the lower bound $\dim
\Mgnd \geq 2g+1$.

We note that although any trigonal curve is canonically embedded in
some rational normal scroll $S_{k,\ell}$ and hence in some
$X(\Delta)_k$, it might fail to be nondegenerate because it can be
impossible to avoid tangency to $X(\Delta)_k \setminus \T^2_k$.

\section*{Acknowledgements}

Part of this paper was written while the first author was supported by the EPSRC grant EP/C014839/1.
He would also like to explicitly thank Marc Coppens for his
helpful comments and for referring us to the Ph.D. thesis of Koelman
\cite{Koelman}, about which we only found out during the final
research phase.  The second author would like to thank Bernd Sturmfels and Steven Sperber.

\section*{Appendix: lattice polytopes of genus one}

There are 16 equivalence classes of lattice polytopes having one
interior lattice point. Polytopes representing these are drawn
below. This is a copy of \cite[Figure 2]{PoonenRodriguez}, we
include the list here for sake of self-containedness. It is an
essential ingredient in the proofs of
Lemma~\ref{genus1nondegcriterion} and Proposition~\ref{maintheorem}.

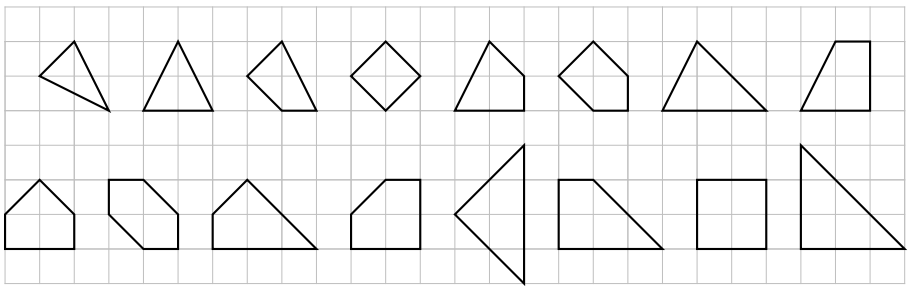
\begin{center}
\psset{unit=0.92cm}
\ \\
\begin{pspicture}(0,0.51)(13,4.4)
\psgrid[gridcolor=lightgray,subgridcolor=lightgray,gridwidth=0.01,subgridwidth=0.01,gridlabels=0,subgriddiv=2]
\pspolygon[fillstyle=none,linecolor=black](0.5,3.5)(1,4)(1.5,3)
\pspolygon[fillstyle=none,linecolor=black](2,3)(3,3)(2.5,4)
\pspolygon[fillstyle=none,linecolor=black](3.5,3.5)(4,4)(4.5,3)(4,3)
\pspolygon[fillstyle=none,linecolor=black](5.5,3)(5,3.5)(5.5,4)(6,3.5)
\pspolygon[fillstyle=none,linecolor=black](6.5,3)(7.5,3)(7.5,3.5)(7,4)
\pspolygon[fillstyle=none,linecolor=black](8,3.5)(8.5,3)(9,3)(9,3.5)(8.5,4)
\pspolygon[fillstyle=none,linecolor=black](9.5,3)(11,3)(10,4)
\pspolygon[fillstyle=none,linecolor=black](11.5,3)(12.5,3)(12.5,4)(12,4)
\pspolygon[fillstyle=none,linecolor=black](0,1)(1,1)(1,1.5)(0.5,2)(0,1.5)
\pspolygon[fillstyle=none,linecolor=black](1.5,1.5)(1.5,2)(2,2)(2.5,1.5)(2.5,1)(2,1)
\pspolygon[fillstyle=none,linecolor=black](3,1)(4.5,1)(3.5,2)(3,1.5)
\pspolygon[fillstyle=none,linecolor=black](5,1)(6,1)(6,2)(5.5,2)(5,1.5)
\pspolygon[fillstyle=none,linecolor=black](6.5,1.5)(7.5,2.5)(7.5,0.5)
\pspolygon[fillstyle=none,linecolor=black](8,1)(9.5,1)(8.5,2)(8,2)
\pspolygon[fillstyle=none,linecolor=black](10,1)(11,1)(11,2)(10,2)
\pspolygon[fillstyle=none,linecolor=black](11.5,1)(13,1)(11.5,2.5)
\end{pspicture}
\end{center}


\begin{thebibliography}{99}

\bibitem{ACGH}
\textsc{E.~Arbarello}, \textsc{M.~Cornalba}, \textsc{P.A.~Griffiths}, and \textsc{J.~Harris}, \emph{Geometry of Algebraic Curves}, Volume~I, Grundlehren der Math. Wiss., vol.~267, Springer-Verlag, New York, 1985.

\bibitem{AS}
\textsc{A.~Adolphson} and \textsc{S.~Sperber}, \emph{Exponential
sums and Newton polyhedra: cohomology and estimates},  Ann.\ of
Math.\ (2) \textbf{130} (1989), no.\ 2, 367--406.


\bibitem{Batyrev}
\textsc{V.~Batyrev}, \emph{Variations of the mixed Hodge structure
of affine hypersurfaces in algebraic tori}, Duke Math.\ J.
\textbf{69} (1993), no.~2, 349--409.

\bibitem{BatyrevCox}
\textsc{V.~Batyrev} and \textsc{D.~Cox}, \emph{On the Hodge
structure of projective hypersurfaces in toric varieties}, Duke
Math.\ J.\ \textbf{75} (1994), no.~2, 293--338.

\bibitem{BatyrevNill}
\textsc{V.~Batyrev} and \textsc{B.~Nill}, \emph{Multiples of
lattice polytopes without interior lattice points}, Mosc.\ Math.\
J.\ \textbf{7} (2007), no.~2, 195--207, 349.

\bibitem{BeelenPellikaan}
\textsc{P.~Beelen} and \textsc{R.~Pellikaan}, \emph{The Newton
polygon of plane curves with many rational points}, Des.\ Codes
Cryptogr.\ \textbf{21} (2000), no.~1-3, 41--67.

\bibitem{BrunsGubeladze}
\textsc{W.~Bruns} and \textsc{J.~Gubeladze}, \emph{Semigroup
algebras and discrete geometry}, Geometry of toric varieties,
S\'emin.\ Congr.\ \textbf{6} (2002), 43--127.


\bibitem{CastryckDenefVercauteren}
\textsc{W.~Castryck}, \textsc{J.~Denef}, and \textsc{F.~Vercauteren}, \emph{Computing zeta functions
of nondegenerate curves}, Int.\ Math.\ Res.\ Pap.\ \textbf{2006}, Article ID 72017, 57 pages.

\bibitem{Clark}
\textsc{P.~Clark}, \emph{There are genus one curves of every index
over every number field}, J.~Reine Angew.~Math.~\textbf{594} (2006),
201--206.

\bibitem{Clark2}
\textsc{P.~Clark}, \emph{On the indices of curves over local fields}, to appear in Manuscripta Math.

\bibitem{Coppens}
\textsc{M.~Coppens}, \emph{The Weierstrass gap sequences of the
ordinary ramification points of trigonal coverings of $\PP^1$;
Existence of a kind of Weierstrass gap sequence}, J. Pure Appl.
Algebra \textbf{43} (1986), 11--25



\bibitem{Fisher}
\textsc{T.~Fisher}, \emph{Invariants of a genus one curve}, \verb|math.NT/0610318|.

\bibitem{Fulton}
\textsc{W.~Fulton}, \emph{Introduction to toric varieties}, Annals of Math.~Studies, vol.~131, William H.\ Roever Lectures in Geometry, Princeton University Press, Princeton, 1993.



\bibitem{GKZ}
\textsc{I.~Gel'fand}, \textsc{M.~Kapranov}, and \textsc{A.~Zelevinsky},
\emph{Discriminants, resultants, and multidimensional determinants}, Mathematics: Theory \& Applications, Birkh\"auser Boston, 1994.

\bibitem{Pick}
\textsc{B.~Gr\"unbaum} and \textsc{G.~Shephard}, \emph{Pick's theorem}, Amer.\ Math.\ Monthly \textbf{100} (1993), no.\ 2, 150--161.

\bibitem{HaaseSchicho}
\textsc{C.~Haase} and \textsc{J.~Schicho}, \emph{Lattice polygons and the number $2i + 7$}, to appear in Amer.\ Math.\ Monthly.

\bibitem{HarrisMorrison}
\textsc{J.~Harris} and \textsc{I.~Morrison}, \emph{Moduli of curves}, Grad.~Texts in Math., vol.~187, Springer-Verlag, New York, 1998.

\bibitem{Hartshorne}
\textsc{R.~Hartshorne}, \emph{Algebraic geometry}, Grad.~Texts in Math., vol.~52, Springer-Verlag, New York, 1977.

\bibitem{Hensley}
\textsc{D.~Hensley}, \emph{Lattice vertex polytopes with interior lattice points}, Pacific J.~Math.\ \textbf{105} (1983), no.~1, 183--191.


\bibitem{Khovanskii}
\textsc{A.~G.~Khovanski\u\i}, \emph{Newton polyhedra, and toroidal varieties}, Functional Anal.\ Appl.\ \textbf{11} (1977), no.~4, 289--296.

\bibitem{Koelman}
\textsc{R.~Koelman}, \emph{The number of moduli of families of curves on toric surfaces}, Proefschrift, Katholieke Universiteit te Nijmegen, 1991.


\bibitem{Kouchnirenko}
\textsc{A.~Kouchnirenko}, \emph{Poly\`edres de Newton et nombres de Milnor}, Inv.\ Math.\ \textbf{32} (1976), 1--31.

\bibitem{KreschWetherell}
\textsc{A.~Kresch}, \textsc{J.~Wetherell}, and \textsc{M.~Zieve}, \emph{Curves of every genus with many points,\ I:\ Abelian and toric families}, J.\ Algebra \textbf{250} (2002), no.\ 1, 353--370.


\bibitem{Maroni}
\textsc{A.~Maroni}, \emph{Le serie lineari sulle curve trigonali},
Ann. Mat. Pura Appl. \textbf{25} (1946), 341--354

\bibitem{Matsumoto}
\textsc{R.~Matsumoto}, \emph{The $C_{ab}$ curve},
\verb|http://www.rmatsumoto.org/cab.ps|.

\bibitem{Mikhalkin}
\textsc{G.~Mikhalkin}, \emph{Real algebraic curves, the moment map and amoebas}, Ann.\ of Math.\ (2) \textbf{151} (2000), no.~1, 309--326.

\bibitem{Mikhalkin2}
\textsc{G.~Mikhalkin}, \emph{Gromov-Witten invariants and tropical algebraic geometry}, to appear.

\bibitem{Miura}
\textsc{S.~Miura}, \emph{Algebraic geometric codes on certain plane curves}, Trans.\ IEICE \textbf{J75-A} (1992), no.~11, 1735--1745.

\bibitem{Mumford}
\textsc{D.~Mumford}, \emph{Geometric Invariant Theory},
Ergebnisse, Springer-Verlag, Heidelberg (1965), vi + 146 pp.

\bibitem{Namba}
\textsc{M.~Namba}, \emph{Geometry of projective algebraic curves}, Monographs and Textbooks in Pure and Appl.~Math., \textbf{88}, Marcel Dekker, New York, 1984.

\bibitem{PoonenRodriguez}
\textsc{B.~Poonen} and \textsc{F.~Rodriguez-Villegas}, \emph{Lattice polygons and the number $12$}, Amer.\ Math.\ Monthly \textbf{107} (2000), no.~3, 238--250.


\bibitem{RimVitulli}
\textsc{D.~Rim} and \textsc{M.~Vitulli}, \emph{Weierstrass points and monomial curves}, J.~Algebra \textbf{48} (1977), 454--476.



\end{thebibliography}
\end{document}